\documentclass[reqno]{amsart}
 \usepackage[letterpaper, left=30mm,right=30mm,top=30mm,bottom=30mm,marginpar=25mm]{geometry}

\usepackage{enumitem}
\usepackage{color}
\usepackage{amsmath}
\usepackage{amssymb}
\numberwithin{equation}{section}
\newtheorem*{assumption*}{Spectral Assumption}
\usepackage{graphicx}

\newtheorem{theorem}{Theorem}

\newtheorem{rem}[theorem]{Remark}
\newtheorem{lemma}[theorem]{Lemma}

\let\eps\varepsilon
\allowdisplaybreaks

\usepackage[normalem]{ulem}

\usepackage{amssymb}

\title[Diffusive limit of the radiative heat transfer system]{Diffusive limits of the steady state radiative heat transfer system: Boundary layers}
\author{Mohamed Ghattassi}
\address{Department of Mathematics, New York University in Abu Dhabi, Saadiyat Island, P.O. Box 129188, Abu Dhabi, United Arab Emirates {\sf mg6888@nyu.edu}}
\author{Xiaokai Huo}
\address{Department of Mathematics, Iowa State University,  411 Morrill Rd, Ames, IA 50011, USA, {\sf xhuo@iastate.edu}}
\author{Nader Masmoudi}
\address{Department of Mathematics, New York University in Abu Dhabi, Saadiyat Island, P.O. Box 129188, Abu Dhabi, United Arab Emirates--
Courant Institute of Mathematical Sciences, New York University, 251 Mercer Street, New York, NY 10012, USA, {\sf nm30@nyu.edu}}

\begin{document}

\begin{abstract}

In this paper, we study the diffusive limit of the steady state radiative heat transfer system for non-homogeneous Dirichlet boundary conditions in a bounded domain with flat boundaries. A composite approximate solution is constructed using asymptotic analysis taking into account of the boundary layers. The convergence to the approximate solution in the diffusive limit is proved using a Banach fixed point theorem. The major difficulty lies on the nonlinear coupling between elliptic and kinetic transport equations. To overcome this problem, a spectral assumption ensuring the linear stability of the boundary layers is proposed. Moreover, a combined $L^2$-$L^\infty$ estimate and the Banach fixed point theorem are used to obtain the convergence proof. This results extend our previous work \cite{ghattassi2020diffusive} for the well-prepared boundary data case to the ill-prepared case when boundary layer exists.

Dans cet article, nous étudions la limite de diffusion du système de transfert de chaleur radiatif en régime stationnaire pour des conditions aux limites de type Dirichlet non homogènes dans un domaine borné avec frontière plate. Une solution approchée composite est construite à l'aide d'une analyse asymptotique prenant en compte les couches limites. La convergence vers la solution approchée dans la limite de diffusion est démontrée à l'aide d'un théorème de point fixe de Banach. La difficulté majeure réside dans le couplage non linéaire entre l'équation elliptique et l'équation de transport cinétique. Pour remédier à ce problème, une hypothèse spectrale assurant la stabilité linéaire des couches limites est proposée. De plus, une estimation combinée $L^{2}-L^{\infty}$ et le théorème du point fixe de Banach sont utilisés pour obtenir la preuve de convergence. Ces résultats étendent nos travaux précédents \cite{ghattassi2020diffusive}  pour le cas des données aux limites bien préparées au cas mal préparé lorsque la couche limite existe.

\end{abstract}


\maketitle   
\tableofcontents
\section{Introduction}
\subsection{Problem Statements}
We consider the following steady state radiative heat transfer system in the space $x\in\Omega = [0,1]\times \mathbb{T}^2$ and $\beta \in \mathbb{S}^2$:
\begin{align}
    &\eps^2 \Delta T^\eps + \langle \psi^\eps - (T^\eps)^4 \rangle =0,\label{eq:1}\\
    &\eps \beta\cdot \nabla \psi^\eps + \psi^\eps - (T^\eps)^4 =0,\label{eq:2}
\end{align} 
with Dirichlet boundary conditions 
\begin{align}
    & T^\eps(x) = T_{b}(x),\text{ for } x\in \partial\Omega = \{0,1\}\times\mathbb{T}^2, \label{eq:1b}\\
     &\psi^\eps(x,\beta) = \psi_{b}(x,\beta),\text{ for } (x,\beta) \in \Gamma_-.\label{eq:2b}
 \end{align}
Here $T^\eps=T^\eps(x)$ is the temperature, $\psi^\eps=\psi^\eps(x,\beta)$ is the radiation intensity. The bracket $\langle \cdot \rangle$ denotes the momentum $\langle \psi(\beta)\rangle = \int_{\mathbb{S}^2}\psi(\beta) d\beta$.
We denote the boundary set $\Gamma$ by 
\begin{align*}
    \Gamma = \{(x,\beta):x\in\partial\Omega, \beta\in\mathbb{S}^2\},
\end{align*}
and $\Gamma_+ = \Gamma \cap \{(x,\beta):\beta\cdot n(x)>0\}$ being the out-flow boundary, and $\Gamma_- = \Gamma \cap \{(x,\beta):\beta\cdot n(x) <0\}$ being the in-flow boundary, where $n(x)$ is the exterior normal vector on the boundary. 
Note that the boundary conditions are imposed on the in-flow boundary and the values of $\psi$ on the out-flow boundary are determined from the system.

When the boundary data is well-prepared, i.e. $\psi_b(x,\beta)=T_b^4(x)$ for $(x,\beta)\in\Gamma_-$, the solution of \eqref{eq:1}-\eqref{eq:2} was proved in \cite{ghattassi2020diffusive} to converge to the following nonlinear elliptic equation 
\begin{align}
    &\Delta T_0 + \frac{4\pi}{3} \Delta T_0^4=0,\label{eq:e10}\\
    &\psi_0 = T_0^4,\label{eq:e20}
\end{align}
subject to boundary conditions 
\begin{align}\label{eq:e0b}
    T_0(x)=T_b(x),\quad \text{for any }x\in\partial\Omega.
\end{align}
The proof is given by using two methods; the weak convergence method and the relative entropy method. However, both methods rely on the assumption $\psi_b=T_b^4$ to obtain the estimates that are needed for the convergence estimates and fail for general boundary data, due to the presence of boundary layers.

The main objective of this paper is to study the diffusive limit ($\eps\to 0$) for the 
general boundary data. First, a boundary layer problem is used to describe the behavior of solutions to system \eqref{eq:1}-\eqref{eq:2} near the boundary as $\eps\to 0$. Second, a composite approximate solution is constructed by combining the solutions to the boundary layer problem and solutions to system \eqref{eq:e10}-\eqref{eq:e20}. In particular, higher order approximate solutions
are constructed via asymptotic analysis. Last, the convergence of solutions to system \eqref{eq:1}-\eqref{eq:2} to the constructed approximate solution is proved by using an $L^2$-$L^\infty$ type estimate.

The boundary layer problem has already been studied in \cite{Bounadrylayer2019GHM2}. 
Here for simplicity of notations, we assume that the boundary data at the top boundary $\{x_1=1\}$ is well-prepared, i.e. $\psi_b(x,\beta) = T_b^4(x)$ for $(x,\beta)\in \Gamma_-\cap \{x_1=1\}$, so that boundary layer only exists at the bottom $\{x_1=0\}$. Introducing $\eta= x_1/\eps^2$, the corresponding boundary layer problem for system \eqref{eq:1}-\eqref{eq:2} is
\begin{align}
   &  \partial_\eta^2 \tilde{T}_0 + \langle \tilde{\psi}_0 - \tilde{T}_0^4 \rangle = 0, \label{eq:ml1}\\
   &\mu \partial_\eta \tilde{\psi}_0 + \tilde{\psi}_0 - \tilde{T}_0^4 = 0,\label{eq:ml2}
\end{align}
with boundary conditions 
\begin{align}
    &\tilde{T}_0(\eta=0,x') = T_b(x'),\quad \text{for any }x'\in\mathbb{T}^2,\label{eq:ml1b}\\
    &\tilde{\psi}_0(\eta=0,x',\beta) = \psi_b(0,x',\beta),\quad \text{for any } x'\in\mathbb{T}^2,\, (0,x',\beta) \in \Gamma_-,\label{eq:ml2b}
\end{align}
where $\mu=-n(x)\cdot\beta = \beta_1$ and $\tilde{T}_0=\tilde{T}_0(\eta,x')$, $\tilde{\psi}_0=\tilde{\psi}_0(\eta,x',\beta)$ with $x'=(x_2,x_3)\in\mathbb{T}^2$. The above problem is also called \emph{nonlinear Milne problem} of the radiative heat transfer system. Existence of weak solutions for this problem was proved in \cite{Bounadrylayer2019GHM2}. Moreover, the weak solutions are shown to converge  as $\eta\to\infty$ to some non-negative constants $\tilde{T}_{0,\infty}(x'):=\lim_{\eta\to\infty} \tilde{T}_0(\eta,x')$, $\tilde{\psi}_{0,\infty}(x',\beta):=\lim_{\eta\to\infty} \tilde{\psi}_0(\eta,x',\beta)$, which give the boundary conditions for the nonlinear limiting equation \eqref{eq:e10} by
\begin{align}\label{eq:T0in}
    T_0(x_1=0,x') = \tilde{T}_{0,\infty}(x'),\quad \text{for any }x'\in\mathbb{T}^2.
\end{align}
However, the linear stability for the nonlinear Milne problem holds under a spectral assumption, which reads as 
\begin{enumerate}[label=({\bf\Alph*})]
\item  \label{asA} There exists a constant $\tau>0$ such that the function $\tilde{T}_0 \in L^2_{\rm loc}(\mathbb{R}_{+}\times \mathbb{T}^2)$ satisfies the inequality
\begin{align}\label{eq:sp}
  \quad M\int_0^\infty e^{2\tau x}(2\tilde{T}_0^{\frac32})^2 |\partial_x f|^2 dx  \ge  4\int_0^\infty  e^{2\tau x} |\partial_x (2\tilde{T}_0^{\frac32})|^2 f^2 dx 
\end{align}
 for any measurable function $f \in C^{1}(\mathbb{R}_{+})$ with $f(0)=0$, for some constant $M<1$.  
\end{enumerate}

A composite approximate solution can thus be constructed by combining  the interior approximate problem \eqref{eq:e10}-\eqref{eq:e20} with \eqref{eq:T0in} and the solution to the Milne problem \eqref{eq:ml1}-\eqref{eq:ml2b}. To do this,
we introduce a cut-off function $\chi=\chi(x_1)=\chi(\eps\eta)$ such that
\begin{align}\label{eq:cutoff}
    \chi(x_1) = \left\{\begin{array}{cl}
        1,&\text{for } 0\le x_1\le \frac14 \delta,\\
        0,&\text{for }x_1>\frac38\delta,\\
        \in (0,1), & \text{otherwise.}
    \end{array}\right.
\end{align}
Here where $\delta>0$ is a small constant and will be chosen later (see Theorem \ref{thm-ap}).  The boundary layer corrections are given by
\begin{align}\label{eq:bar0def2}
    \bar{T}_0 = \chi(x_1)(\tilde{T}_0-\tilde{T}_{0,\infty}),\quad \bar{\psi}_0 = \chi(x_1)(\tilde{\psi}_0-\tilde{T}_{0,\infty}^4).
\end{align}
The major goal of this work is to prove that $(T_0+\bar{T}_0,\psi_0+\bar{\psi}_0)$ provides an approximate solution to system \eqref{eq:1}-\eqref{eq:2} when $\eps$ is small.

\subsection{Main results}
The main result of this paper is the following theorem.
\begin{theorem}\label{thm.1}
Let  $T_b \in C^2(\partial\Omega)$ and $\psi_b \in C^1(\Gamma_-)$ be non-negative functions. Let $(T_0,\psi_0)$ be the smooth solution to system \eqref{eq:e10}-\eqref{eq:e20} with boundary condition \eqref{eq:T0in} and $(\tilde{T}_0,\tilde{\psi}_0)$ be the smooth solution to the nonlinear Milne problem \eqref{eq:ml1}-\eqref{eq:ml2b}. Let the boundary layer correction $(\bar{T}_0,\bar{\psi}_0)$ be given by \eqref{eq:bar0def2}.
 Assume $\tilde{T}_0$ satisfies the spectral assumption \ref{asA} and has a lower bound $\tilde{T}_0\ge a$ for some constant $a>0$. Then for $\eps>0$ sufficiently small, there exists a unique solution $(T^\eps,\psi^\eps) \in L^\infty(\Omega) \times L^\infty(\Omega\times\mathbb{S}^2)$ to system \eqref{eq:1}-\eqref{eq:2} with boundary conditions \eqref{eq:1b}-\eqref{eq:2b}, satisfying
    \begin{align}\label{eq:thm}
        &\|T^\eps - T_0 - \bar{T}_0 \|_{L^\infty(\Omega)} = O(\eps),\quad
        \|\psi^\eps - T_0^4 - \bar{\psi}_0\|_{L^\infty(\Omega\times\mathbb{S}^2)} = O(\eps).
    \end{align}
\end{theorem}

The above theorem rigorously shows the convergence of solutions of system \eqref{eq:1}-\eqref{eq:2} to the approximate solution $(T_0+\bar{T}_0)$ as $\eps\to 0$.
In order to prove the above theorem, we need higher order approximate expansions beyond \eqref{eq:e10}-\eqref{eq:e20} and \eqref{eq:ml1}-\eqref{eq:ml2}. 
Let $N\ge 1$ be an integer, we take the ansatz for the composite approximate solution $(T^a,\psi^a)$ as
\[
T^\eps \sim T^a:=\sum_{k=0}^N\eps^k(T_k + \bar{T}_k), \quad \psi^\eps \sim \psi^a:=  \sum_{k=0}^N \eps^k(\psi_k + \bar{\psi}_k),
\] 
where $(T_k,\psi_k)$ is the $k$-th order interior approximation (satisfying system \eqref{eq:otk}-\eqref{eq:opsik}) and $(\bar{T}_k,\bar{\psi}_k)$ is the $k$-th order boundary layer correction (defined by \eqref{eq:barkdef} and system \eqref{eq:bTk}-\eqref{eq:bpsik}), see Section 2 for details of the derivations. Note that due to the nonlinearity of $(T^\eps)^4$, the interior expansion and boundary layer expansion are coupled. In particular, Taylor's expansions of the interior approximations are used for the  construction of the boundary layer corrections.
The composite approximate solution could be shown to satisfy
\begin{align}
    &\eps^2 \Delta T^a + \langle \psi^a - (T^a)^4 \rangle =\mathcal{R}_1(T^a,\psi^a), \label{eq:1a}\\
    &\eps \beta\cdot \nabla \psi^a + \psi^a - (T^a)^4 = \mathcal{R}_2(T^a,\psi^a),\label{eq:2a}
\end{align}
with boundary conditions 
\begin{align}
    &T^a(x) = T_b(x),\quad \text{for any }x\in \partial\Omega,\label{eq:1ab}\\
    &\psi^a(x,\beta) = \psi_b(x,\beta),\text{ for } (x,\beta)\in \Gamma_-,\label{eq:2ab}
\end{align}
where $\mathcal{R}_1(T^a,\psi^a)$ and $\mathcal{R}_2(T^a,\psi^a)$ are the approximation errors and are of order $\eps^{N+1}$ (see  Theorem \ref{thm-ap}). 

The main idea for the proof of Theorem \ref{thm.1} is to show the existence and uniqueness of system \eqref{eq:1}-\eqref{eq:2} in a small neighborhood of the approximate solution $(T^a,\psi^a)$. In order to achieve this, we construct a sequence of functions $\{T^k,\psi^k\}_{k=0}^\infty$ solving 
\begin{align}
    &\eps^2\Delta T^k + \langle \psi^k - 4(T^a)^3 T^k\rangle = \langle (T^{k-1})^4 - 4(T^a)^3T^{k-1}\rangle,\label{eq:tt1}\\
    &\eps \beta\cdot \nabla \psi^k +  \psi^k - 4(T^a)^3 T^k = (T^{k-1})^4 - 4(T^a)^3T^{k-1},\label{eq:tt2}
\end{align}
with boundary conditions 
\begin{align*}
    &T^k(x) = T_b(x),\quad \text{for any }x\in \partial\Omega,\\
    &\psi^k(x,\beta) = \psi_b(x,\beta),\text{ for } (x,\beta)\in \Gamma_-.
\end{align*}
This defines a mapping $\mathcal{T}: L^\infty(\Omega)\times L^\infty(\Omega \times\mathbb{S}^2) \to  L^\infty(\Omega)\times L^\infty(\Omega \times\mathbb{S}^2)$ with $(T^k,\psi^k) =\mathcal{T}((T^{k-1},\psi^{k-1}))$. For $\eps>0$ sufficiently small and $N\ge 5$, $s\ge 3$, the mapping $\mathcal{T}$ can be shown to be a contraction  in the space 
\begin{align}
    \mathcal{O}_s:=&\{(T,\psi)\in L^2\cap L^\infty(\Omega)\times L^2\cap L^\infty(\Omega\times\mathbb{S}^2):\nonumber\\
    &\qquad \|T-T^a\|_{L^2\cap L^\infty(\Omega)}\le K \eps^s, \|\psi-\psi^a\|_{L^2\cap L^\infty(\Omega\times\mathbb{S}^2)}\le K\eps^s\},
\end{align}
where $K>0$ is a positive constant. Then by the Banach fixed point theorem, there exists a fixed point $(T^\eps,\psi^\eps)$ of $\mathcal{T}$  such that $\mathcal{T}((T^\eps,\psi^\eps)) = (T^\eps,\psi^\eps)\in \mathcal{O}_s$. By the definition of $\mathcal{T}$, $(T^\eps,\psi^\eps)$ solves system \eqref{eq:1}-\eqref{eq:2} with boundary conditions \eqref{eq:1b}-\eqref{eq:2b}.
For example, we can take $s=3$ and $N=5$ in the above arguments and conclude that there exists a unique solution $(T^\eps,\psi^\eps) \in \mathcal{O}_{s=3}$, i.e.
\begin{align}
    &\left\|T^\eps - \sum_{k=0}^5 \eps^k T_k - \sum_{k=0}^{5}\eps^k \bar{T}_k \right\|_{L^2\cap L^\infty(\Omega)} \le C \eps^3, \label{eq:111}\\
    &\left\|\psi^\eps - \sum_{k=0}^5 \eps^k \psi_k - \sum_{k=0}^{5} \eps^k\bar{\psi}_k \right\|_{L^2\cap L^\infty(\Omega\times\mathbb{S}^2)} \le C \eps^3,\label{eq:112}
\end{align}
which leads to \eqref{eq:thm} and proves Theorem \ref{thm.1}.

One of the most elusive and difficult issues to prove Theorem \ref{thm.1} is to  show that $\mathcal{T}$ is a contraction mapping.
This is solved by using an $L^2$-$L^\infty$ estimates on system \eqref{eq:tt1}-\eqref{eq:tt2}. In order to get an $L^2$ estimate on this system, the following coercivity inequality
\begin{align}\label{eq:est/gg}
    -\int_\Omega 4(T^a)^3 g\Delta g dx \ge - C \|g\|_{L^2(\Omega)}^2,
\end{align}
for any function $g \in H^1(\Omega)$,
is used, which can be shown to hold under the spectral assumption \ref{asA} (see Lemma \ref{lm.g2}). The $L^\infty$ estimate is derived based on the elliptic regularity and the maximum principle for the radiative transport equation.

The spectral assumption \ref{asA} plays a key role in the proof of Theorem \ref{thm.1}. First, it's required to show the well-posedness of the nonlinear Milne problem \eqref{eq:ml1}-\eqref{eq:ml2b} and equations for higher order boundary corrections $(\bar{T}_k,\bar{\psi}_k)$, $k\ge 1$. Second, under this assumption, the exponential decay of $\bar{T}_k,\bar{\psi}_k$, $k\ge 0$ can be shown, which is needed in order to find the boundary condition for the interior approximations. Third, the spectral assumption is used to show the inequality \eqref{eq:est/gg}, which is crucial in order to get the suitable $L^2$ estimate on system \eqref{eq:tt1}-\eqref{eq:tt2} and prove $\mathcal{T}$ is a contraction mapping. Note since $\bar{T}_k=\bar{T}_k(\eta,x')=\bar{T}(x_1/\eps,x')$, $\partial^2_{\eta}\bar{T}_k$ may be of order $1/\eps^2$ and thus the left term of \eqref{eq:est/gg} may be singular as $\eps\to 0$. Thanks to the spectral assumption \ref{asA}, inequality \eqref{eq:est/gg} holds and overcome the above singularity.

Our work also provides another approach to justify the diffusive limit in the well-prepared case, which was already done in \cite{ghattassi2020diffusive} using different methods. Indeed, when the boundary data is well-prepared, $\psi_b=T_b^4$, and no boundary layer exists. In this case, the boundary layer corrections $\bar{T}_k \equiv 0, \bar \psi_k \equiv 0$. We can thus take $T^a=\sum_{k=0}^N \eps^k T_k$, $\psi^a=\sum_{k=0}^N \eps^k \psi_k$. Since $\nabla T^a$ is bounded, inequality \eqref{eq:est/gg} holds obviously. Therefore, Theorem \ref{thm.1} holds and \eqref{eq:thm} implies the convergence of $(T^\eps,\psi^\eps)$ to $(T_0,\psi_0)$, which is the solution to the problem \eqref{eq:e10}-\eqref{eq:e0b}. 

\subsection{Related work}
When the diffusion operator is not considered, system \eqref{eq:1}-\eqref{eq:2} reduces to the linear transport equation $\eps \beta \cdot \nabla U^\eps + U^\eps - \langle U^\eps\rangle/(4\pi) = 0$. As $\eps \to 0$, its solution convergences to $U_0 + \bar{U}_0$, where $U_0=U_0(x)$ is the solution of the Laplacian equation $\Delta U_0=0$ and $\bar{U}_0=\bar{U}_0(x,\beta)$ is a boundary layer correction defined by $\bar{U}_0=\chi(x_1)(\tilde{U}_0 -\tilde{U}_{0,\infty})$ where $\tilde{U}_0=\tilde{U}(\eta,x',\beta)$ is the solution to $\mu\partial_\eta \tilde{U}_0 + \tilde{U}_0 - \langle \tilde{U}_0\rangle/(4\pi)=0$. This convergence was first rigorously proved in \cite{bensoussan1979boundary} whereas the study of this linear Milne problem was done in \cite{bardos1984diffusion}. However, when the boundary is not flat, $\|U^\eps - U_0 - \bar{U}_0\|_{L^\infty}$ is shown to not converge to zero as $\eps\to 0$ in \cite{wu2015geometric}. This can be overcome by introducing a geometric correction $U_0^\eps,\bar{U}_0^\eps$ with considerations of the curvature effects. They proved that $\|U^\eps - U_0^\eps-\bar{U}_0^\eps\|_{L^\infty}$ converges to zero as $\eps\to 0$ in the 2D unit disk \cite{wu2015geometric}, in the annulus \cite{wu2016asymptotic}, in the 2D convex domain \cite{guo2017geometric,wu2020asymptotic2} and in the 3D convex domain \cite{wu2021diffusive}. For more references on the diffusive limit of the linear transport equation, we refer the reader to \cite{li2017validity,wu2020asymptotic2} and references therein.


 When the Laplacian term $\eps^2\Delta T^\eps$ is neglected, the diffusive limit for system \eqref{eq:1}-\eqref{eq:2} has been studied in many works \cite{golse1985b,bardos1987rosseland,bardos1988nonaccretive,clouet1997rosseland,larsen1983asymptotic}. For example in \cite{bardos1987rosseland}, the authors derived the Rosseland approximation on a different radiative transfer equation depending on the frequency variable. The strong convergence of the solution of the radiative transfer equation to the solution of the Rosseland equation for well-prepared boundary data is proved using the so-called Hilbert’s expansion method. Therefore, in \cite{bardos1988nonaccretive}, under some weak hypotheses on the various parameters of the radiative transfer equation, the Rosseland approximation was justified in a weak sense. Moreover, the boundary layer problem for system \eqref{eq:1}-\eqref{eq:2} without the Laplacian term is constructed in \cite{larsen1983asymptotic} and the boundary condition for the limiting system \eqref{eq:e10} is shown to satisfy a mixed Robin boundary condition. Such a method is extended in \cite{klar1998boundary} to construct the boundary layer approximations for system \eqref{eq:1}-\eqref{eq:2}. However, the method only provides the zeroth and first order approximations near the boundary and could not extend to get higher order approximations. Since our estimates \eqref{eq:111}-\eqref{eq:112} need higher order expansions, here we provide a different way to obtain the approximation boundary layer solutions up to any order.

\subsection{Plan of the paper} The paper is organized as follows: In the next section, we construct the composite approximate solutions and show they satisfies the radiative heat transfer system in the perturbative sense, in Lemma \ref{lem.ap}. In section \ref{sec3}, the properties of the approximate solutions are studied and the approximation errors are shown in Theorem \ref{thm-ap}, and inequality \eqref{eq:est/gg} is proved in Lemma \ref{lm.g2}. Finally, the proof of Theorem \ref{thm.1} is given, which consists of showing the linearized stability and nonlinear stability of \eqref{eq:1}-\eqref{eq:2} in the neighborhood of the approximate solution.

\textbf{Notations.}  Throughout the paper, some standard notations are used. The norm $\|\cdot \|_{L^2(\Omega)}$ and $\|\cdot \|_{L^2(\Omega\times\mathbb{S}^2)}$ are defined by 
${\|f\|_{L^2(\Omega)}^2 = \int_\Omega f^2 dx,\,\,\forall f\in L^2(\Omega)}$ and $\|g\|_{L^2(\Omega\times\mathbb{S}^2)}^2 = \int_{\Omega\times\mathbb{S}^2} g^2 d\beta dx,\,\,\,\forall g\in L^2(\Omega\times\mathbb{S}^2)$.
The norm $\|\cdot \|_{L^2(\Gamma_+)}$ and $\|\cdot\|_{L^2(\Gamma_-)}$ are defined respectively by $\|g\|_{L^2(\Gamma_+)}^2= \int_{\Gamma_+}\beta\cdot n(x) g^2  d\beta d\sigma_x$, and $\|g\|_{L^2(\Gamma_-)}^2= \int_{\Gamma_-}|\beta\cdot n(x)| g^2  d\beta d\sigma_x$, where $\sigma_x$ is the surface element.
\section{Asymptotic Analysis}\label{sec2}
In this section, an approximate solution to system \eqref{eq:1}-\eqref{eq:2} is constructed via asymptotic analysis. An interior expansion is first constructed which is valid in the interior of the domain and then boundary layer corrections are constructed accounting for the boundary layer effects. Finally, we combine the results to get a composite approximate solution to system \eqref{eq:1}-\eqref{eq:2}. We here recall the simplifying assumption that the boundary layer only occurs near the bottom $\{x\in \partial \Omega,x_1=0\}$.

\subsection{Interior expansion}
We take the interior expansion to be 
\begin{align}\label{eq:ansatz}
    T^\eps(x) \sim \sum_{k=0}^N \eps^k T_k(x),\quad \psi^\eps(x,\beta) \sim \sum_{k=0}^N \eps^k \psi_k(x,\beta).
\end{align}
Define 
\begin{align}\label{eq:r1r2}
    \mathcal{R}_1(T,\psi) :=\eps^2 \Delta T + \langle \psi - T^4 \rangle,\quad
    \mathcal{R}_2(T,\psi) :=\eps \beta\cdot \nabla \psi + (\psi-T^4).
\end{align}
Plugging \eqref{eq:ansatz} into the above formulas gives 
\begin{align}
    &\mathcal{R}_1\left(\sum_{k=0}^N \eps^k T_k,\sum_{k=0}^N \eps^k \psi_k\right)\nonumber\\
    &\quad = \sum_{k=0}^N \eps^k (\Delta T_{k-2} + \langle \psi_k - \mathcal{C}(T,k)) \rangle + \eps^{N+1} \Delta T_{N-1} + \eps^{N+2} \Delta T_N - \sum_{k=N+1}^{4N}\varepsilon^k \langle \mathcal{C}(T,k)\rangle ,\label{eq:R/1}\\
    &\mathcal{R}_2 \left(\sum_{k=0}^N \eps^k T_k,\sum_{k=0}^N \eps^k \psi_k\right) \nonumber\\
    &\quad = \sum_{k=0}^N \eps^k (\beta\cdot \nabla \psi_{k-1}+\psi_k - \mathcal{C}(T,k)) + \eps^{N+1} \beta\cdot \nabla \psi_{N} - \sum_{k=N+1}^{4N}\varepsilon^k  \mathcal{C}(T,k) ,\label{eq:R/2}
\end{align}
where 
\begin{align}\label{eq:mathcalC}
    \mathcal{C}(T,k) := \sum_{\substack{i+j+l+m=k\\i,j,l,m\ge 0}} T_iT_jT_lT_m,
\end{align}
and $(T_{k}, \psi_k), k<0$ are taken to be zero.

Collecting terms with the same order, we take
\begin{align}
    &\Delta T_{k-2} + \langle \psi_k - \mathcal{C}(T,k) \rangle = 0,\label{eq:o/1}\\
    &\beta\cdot \nabla \psi_{k-1} + \psi_k - \mathcal{C}(T,k) = 0,\label{eq:o/2}
\end{align}
 for any $k=0,\ldots, N$. Therefore from the above two equations we obtain
\begin{align*}
    \Delta T_{k-2} = \langle \beta\cdot \nabla \psi_{k-1} \rangle.
\end{align*}
By \eqref{eq:o/2}, an iterative process on the above equation leads to 
\begin{align*}
    \Delta T_{k-2} &= \langle \beta\cdot \nabla (-\beta\cdot \nabla \psi_{k-2}+\mathcal{C}(T,k-1))\rangle = -\langle \beta\cdot \nabla (\beta\cdot \nabla \psi_{k-2})\rangle \nonumber\\
    &= -\langle (\beta\cdot \nabla)^2 (-\beta\cdot\nabla \psi_{k-3}+\mathcal{C}(T,k-2))\rangle = -\frac{4\pi}{3} \Delta \mathcal{C}(T,k-2)\rangle + \langle (\beta\cdot\nabla)^3\psi_{k-3}\rangle.
\end{align*}
Consequently, equations \eqref{eq:o/1}-\eqref{eq:o/2} can be rewritten as 
\begin{align}
        &\Delta T_k + \frac{4\pi}{3} \Delta \mathcal{C}(T,k) = \langle (\beta\cdot\nabla)^3 \psi_{k-1}\rangle, \label{eq:o/1-1} \\
    &\psi_k = -\beta\cdot \nabla \psi_{k-1} + \mathcal{C}(T,k),\label{eq:o/2-1} 
\end{align}
for any $k=0,\ldots,N$.
 
Note that here $T_k$ is solved by \eqref{eq:o/1-1} and the the solution could be plugged into \eqref{eq:o/2-1} to get $\psi_k$. In addition, for $k=0$, equation \eqref{eq:o/1-1} becomes
\begin{align*}
    \Delta T_0 + \frac{4\pi}{3} \Delta T_0^4 =0,
\end{align*}
which is a nonlinear elliptic equation. For $k\ge 1$, equation \eqref{eq:o/1-1} can be rewritten as 
\begin{align*}
    \Delta T_k + \frac{4\pi}{3} \Delta (4T_0^3 T_k) = - \frac{4\pi}{3} \Delta \mathcal{E}(T,k-1) + \langle (\beta\cdot\nabla)^3 \psi_{k-1}\rangle,
\end{align*}
which is a linear elliptic equation, where 
\begin{align*}
    \mathcal{E}(T,k-1): = \sum_{\substack{i+j+l+m=k\\i,j,l,m\ge 1}} T_iT_jT_lT_m.
\end{align*}
From the above equations, the residuals \eqref{eq:R/1}-\eqref{eq:R/2} is given by
\begin{align*}
    &\mathcal{R}_1\left(\sum_{k=0}^N \eps^k T_k,\sum_{k=0}^N \eps^k \psi_k\right) =  \eps^{N+1} \Delta T_{N-1} + \eps^{N+2} \Delta T_N - \sum_{k=N+1}^{4N}\varepsilon^k \langle \mathcal{C}(T,k)\rangle,\\
    &\mathcal{R}_2\left(\sum_{k=0}^N \eps^k T_k,\sum_{k=0}^N \eps^k \psi_k\right) = \eps^{N+1} \beta\cdot \nabla \psi_{N} - \sum_{k=N+1}^{4N}\varepsilon^k  \mathcal{C}(T,k),
\end{align*}
where the right hand side are both formally of order $\varepsilon^{N+1}$.

\subsection{Boundary layer corrections}

We next find the approximation of system \eqref{eq:1}-\eqref{eq:2} near the boundary. Let $\eta=x_1/\varepsilon$, we take the ansatz 
\begin{align*}
    T^\varepsilon(x) \sim \sum_{k=0}^{N} \varepsilon^k (\bar{T}_k(\eta,x') + T_k(x)),\quad \psi^\varepsilon(x,\beta) \sim \sum_{k=0}^N \varepsilon^k (\bar{\psi}_k(\eta,x',\beta) + \psi_k(x,\beta)),
\end{align*}
where $\bar{T}_k,\bar{\psi}_k$ are the correction terms near the boundary and $(T_k,\psi_k)$ are the interior expansions derived in the previous subsection.
Taking the above ansatz into  \eqref{eq:r1r2} gives
\begin{align}\label{eq:R1p}
    &\mathcal{R}_1\left(\sum_{k=0}^N \eps^k(T_k+\bar{T}_k),\sum_{k=0}^N \eps^k(\psi_k + \bar{\psi}_k)\right)\nonumber\\
    &\quad=\mathcal{R}_1\left(\sum_{k=0}^N \eps^k T_k,\sum_{k=0}^N \eps^k \psi_k\right) - \sum_{k=0}^N \eps^k \langle \mathcal{C}(T+\bar{T},k) - \mathcal{C}(T,k)\rangle - \sum_{k=N+1}^{4N} \eps^k\langle \mathcal{C}(T+\bar{T},k) - \mathcal{C}(T,k)\rangle \nonumber\\
    &\qquad + \varepsilon^{N+1} \Delta_{x'}\bar{T}_{N-1} + \eps^{N+2}\Delta_{x'} \bar{T}_N + \sum_{k=0}^N \eps^k \langle \mathcal{C}(\bar{T}+P,k) - \mathcal{C}(P,k) \rangle\nonumber\\
    &\qquad + \sum_{k=0}^N \eps^k (\partial_\eta^2 \bar{T}_k + \Delta_{x'}\bar{T}_{k-2} + \langle \bar{\psi}_k - \mathcal{C}(\bar{T}+P,k) + \mathcal{C}(P,k) \rangle) .
\end{align}
and 
\begin{align}\label{eq:R2p}
    &\mathcal{R}_2\left(\sum_{k=0}^N \eps^k(T_k+\bar{T}_k),\sum_{k=0}^N \eps^k(\psi_k + \bar{\psi}_k)\right)\nonumber\\
    &\quad=\mathcal{R}_2\left(\sum_{k=0}^N \eps^k T_k,\sum_{k=0}^N \eps^k \psi_k\right) -\sum_{k=0}^N \eps^k (\mathcal{C}(T+\bar{T},k) - \mathcal{C}(T,k)) - \sum_{k=N+1}^{4N} \eps^k( \mathcal{C}(T+\bar{T},k) - \mathcal{C}(T,k)) \nonumber\\
    &\qquad + \varepsilon^{N+1} \beta'\cdot\nabla_{x'} \bar{\psi}_N + \sum_{k=0}^N \eps^k ( \mathcal{C}(\bar{T}+P,k) - \mathcal{C}(P,k) )  \nonumber\\
    &\qquad +\sum_{k=0}^N \eps^k (\mu \partial_\eta \bar{\psi}_k + \beta'\cdot\nabla_{x'}\bar{\psi}_{k-1} +  \bar{\psi}_k - \mathcal{C}(\bar{T}+P,k) + \mathcal{C}(P,k)) .
\end{align}
Here $\mu=\beta_1$, $x'=(x_2,x_3)$, $\beta'=(\beta_2,\beta_3)$, $T_k(0) =T_k(x_1=0)$ and $P_k=P_k(\eta,x'),k=0,\ldots,N$ are the Taylor's expansion of $T_k(\eps \eta,x')$ around $\eta=0$, which is given by 
\begin{align}\label{eq:Pkdef}
    P_k(\eta,x')= \sum_{l=0}^k \frac{\eta^l}{l!}\frac{\partial^l}{\partial x_1^l} T_{k-l}(0,x').
\end{align}

Collecting terms of the same order in \eqref{eq:R1p}-\eqref{eq:R2p}, we take
\begin{align}
    &\partial_\eta^2 \bar{T}_k + \Delta_{x'} \bar{T}_{k-2} + \langle \bar{\psi}_k - \mathcal{C}(\bar{T} + P,k) + \mathcal{C}(P,k)\rangle = 0,\\
    &\mu \partial_\eta \bar{\psi}_k + \beta'\cdot \nabla_{x'}\bar{\psi}_{k-1} + \bar{\psi}_k - \mathcal{C}(\bar{T} + P,k) + \mathcal{C}(P,k) = 0,
\end{align}
 for $k=0,\ldots,N$. Let $\tilde{T}_k=\bar{T}_k + P_k(0)=\bar{T}_k + T_k(0)$ for any $k=0,\ldots,N$, and $\tilde{\psi}_0 = \bar{\psi}_0 + {T}_0^4(0)$ and $\tilde{\psi}_k = \bar{\psi}_k + 4T_0^3(0)T_k(0) $, then 
\begin{align*}
    -(\bar{T}_0+P_0)^4 + P_0^4 = -\tilde{T}_0^4 + T^4(0),
\end{align*}
and for $k\ge 1$,
\begin{align*}
    -4(\bar{T}_0+P_0)^3(\bar{T}_k+P_k) + 4P_0^3P_k 
    &=-4\tilde{T}_0^3\tilde{T}_k + 4T_0^3(0)T_k(0) - 4(\tilde{T}_0^3 - T_0^3(0))(P_k-P_{k}(0)).
\end{align*}
Therefore,
$\tilde{T}_k$, $\tilde{\psi}_k$ satisfies the equations 
\begin{align}
    \partial_\eta^2 \tilde{T}_0 + \langle \tilde{\psi}_0 - \tilde{T}_0^4 \rangle = 0,\\
    \mu \partial_\eta \tilde{\psi}_0 +  \tilde{\psi}_0 - \tilde{T}_0^4  = 0,
\end{align}
and for $k=1,\ldots,N$,
\begin{align}
    &\partial_\eta^2 \tilde{T}_k + \Delta_{x'} \bar{T}_{k-2} + \langle \tilde{\psi}_k - 4 \tilde{T}_0^3\tilde{T}_k\rangle \nonumber\\
    &\quad  + \langle - 4(\tilde{T}_0^3-T_0^3(0))(P_k-P_k(0)) - \mathcal{E}(\bar{T}+{P}(0),k-1) + \mathcal{E}(P(0),k-1) \rangle =0,\label{eq:tilde1}\\
    &\mu \partial_\eta \tilde{\psi}_k + \beta'\cdot\nabla_{x'} \bar{\psi}_{k-1} + \tilde{\psi}_k \nonumber\\
    &\quad - 4 \tilde{T}_0^3\tilde{T}_k -4(\tilde{T}_0^3-T_0^3(0))(P_k-P_k(0)) - \mathcal{E}(\bar{T}+{P}(0),k-1) + \mathcal{E}(P(0),k-1)  = 0.\label{eq:tilde2}
\end{align}
The boundary conditions for the above equations are 
\begin{align}
    &\tilde{T}_0(\eta=0,x') = T_b(0,x'), \quad \text{for any }x' \in \mathbb{T}^2,\label{eq:tildebd1}\\
    &\tilde{\psi}_0(\eta=0,x',\beta) = \psi_b(0,x',\beta),\quad \text{for any } (0,x',\beta) \in \Gamma_-, \label{eq:tildebd2}
\end{align}
and for $k=1,\ldots,N$,
\begin{align}
    &\tilde{T}_k(\eta=0,x') = 0, \quad \text{for any }x' \in \mathbb{T}^2,\label{eq:tildebd3}\\
    &\tilde{\psi}_{k}(\eta=0,x',\beta) = \psi_k(0) - 4T_0^3(0)T_k(0),\quad \text{for any } (0,x',\beta) \in \Gamma_-. \label{eq:tildebd4}
\end{align}
The boundary conditions above are taken to be consistency with boundary conditions \eqref{eq:1b}-\eqref{eq:2b} such that  
\begin{equation}\label{eq:bconsis}
\begin{aligned}
    &\sum_{k=0}^N (\bar{T}_k(0,x')+T_k(0,x')) = T_b(0,x'),\quad \text{for }x'\in \mathbb{T}^2, \\
    &\sum_{k=0}^N (\bar{\psi}_k(0,x',\beta)+\psi_k(0,x',\beta))=\psi_b(0,x',\beta),\quad \text{for} (0,x',\beta)\in \Gamma_-.
\end{aligned}
\end{equation}

\subsection{Construction of the approximate composition solution}
In order to combine the interior expansions and boundary corrections, we use the cut-off function $\chi(x_1)$ defined in \eqref{eq:cutoff} and we also introduc another cut-off function $\chi_0(x_1)$ by 
\begin{align}\label{eq:cutoff0}
    \chi_0(x_1) = \left\{\begin{array}{cl}
        1,&\text{for } 0 \le x_1 \le \frac12\delta,\\
        0,&\text{for } x \ge \frac34\delta,\\
        \in (0,1), &\text{otherwise.}
    \end{array}\right.
\end{align} 
The construction of the composite approximate solution is done via the following procedure.

\smallskip
\emph{{Step 1. Construction of $(\bar{T}_0,\bar{\psi}_0)$ and $(T_0,\psi_0)$.}} We first solve \eqref{eq:tilde1}-\eqref{eq:tilde2} when $k=0$ :
\begin{align}
    \partial_\eta^2 \tilde{T}_0 + \langle \tilde{\psi}_0 - \tilde{T}_0^4 \rangle =0,\label{eq:bT0}\\
    \mu \partial_\eta \tilde{\psi}_0 + \tilde{\psi}_0 - \tilde{T}_0^4 = 0,\label{eq:bpsi0}
\end{align} 
with boundary conditions
\begin{align}
    &\tilde{T}_0(\eta=0,x') = T_b(0,x'),\quad \text{for any }x'\in\mathbb{T}^2,\label{eq:bT0b}\\
    &\tilde{\psi}_0(\eta=0,x',\beta) = \psi_b(0,x',\beta),\quad \text{for any } x'\in \mathbb{T}^2, \beta\in\mathbb{S}^2 \text{ and } \mu=\beta_1>0. \label{eq:bpsi0b}
\end{align}
This is the nonlinear Milne problem \eqref{eq:ml1}-\eqref{eq:ml2b}. It has been shown in \cite{Bounadrylayer2019GHM2} that the above problem has a global weak solution $(\tilde{T}_0,\tilde{\psi}_0) \in L^2_{\rm loc}(\mathbb{R}_+\times \mathbb{T}^2)\cap L^2_{\rm loc}(\mathbb{R}_+\times\mathbb{T}^2\times\mathbb{S}^2)$ and as $\eta\to \infty$, the solution converges to some bounded functions $(\tilde{T}_{0,\infty},\tilde{\psi}_{0,\infty})$ independent of $\eta$.
 We define the lowest order boundary layer correction $(\bar{T}_0,\bar{\psi}_0)$ by
\begin{align}\label{eq:bar0def}
    \bar{T}_0(\eta,x') = \chi(\varepsilon\eta)(\tilde{T}_0(\eta,x') - \tilde{T}_{0,\infty}(x')),\quad \bar{\psi}_0(\eta,x',\beta) = \chi(\eps\eta)(\tilde{\psi}_0(\eta,x',\beta)- \tilde{\psi}_{0,\infty}(x',\beta)).
\end{align}
Note due to the property of the nonlinear Milne problem, we have $\tilde{\psi}_{0,\infty} = \tilde{T}_{0,\infty}^4.$ Here and below, $\bar{T}_k, \bar{\psi}_k$ are redefined by using a cut-off and is different from the same notation in section 2.2.

We next give the leading order of the interior expansion $(T_0,\psi_0)$, which is obtained by solving \eqref{eq:o/1-1}-\eqref{eq:o/2-1} for $k=0$:
\begin{align}\label{eq:ot0}
    \Delta T_0 + \frac{4\pi}{3} \Delta T_0^4 = 0,\quad \psi_0 = T_0^4
\end{align}
with boundary conditions 
\begin{align}\label{eq:ot0b}
    T_0(0,x') = \tilde{T}_{0,\infty}(x'),\quad T_0(1,x') = T_b(1,x'),\quad \text{for any }x'\in\mathbb{T}^2.
\end{align}

\emph{{Step 2. Construction of $(\bar{T}_k,\bar{\psi}_k)$ and $(T_k,\psi_k)$ for $k=1,\ldots,N$.}} We solve \eqref{eq:tilde1}-\eqref{eq:tilde2} for $1\le k \le N$:
\begin{align}
    &\partial_\eta^2 \tilde{T}_k + \langle \tilde{\psi}_k - 4\tilde{T}_0^3 \tilde{T}_k\rangle = -\chi_0 \Delta_{x'} \bar{T}_{k-2} + \langle 4(\tilde{T}_0^3-\tilde{T}_{0,\infty}^3)(P_k-P_k(0)) \rangle \nonumber\\
    &\quad +\langle \mathcal{E}(\bar{T}+P,k-1) - \mathcal{E}(P,k-1) \rangle , \label{eq:bTk}\\
    &\mu \partial_\eta \tilde{\psi}_k +  \tilde{\psi}_k - 4\tilde{T}_0^3 \tilde{T}_k = -\chi_0\beta'\cdot\nabla_{x'} \bar{\psi}_{k-1}  + 4(\tilde{T}_0^3-\tilde{T}_{0,\infty}^3)(P_k-P_k(0)) \nonumber\\
    &\quad+ (\mathcal{E}(\bar{T}+P,k-1) - \mathcal{E}(P,k-1)), \label{eq:bpsik}
\end{align}
with boundary conditions
\begin{align}
    &\tilde{T}_k(\eta=0,x') = 0,\quad \text{for any }x'\in\mathbb{T}^2,\label{eq:bTkb}\\
    &\tilde{\psi}_k(\eta=0,x',\beta) =  \beta\cdot\nabla \psi_{k-1}(0) - \mathcal{E}(T(0),k-1),\quad \text{for any } x'\in \mathbb{T}^2, \beta\in\mathbb{S}^2 \text{ and } \mu=\beta_1>0. \label{eq:bpsikb}
\end{align}
The above problem was studied in \cite{Bounadrylayer2019GHM2} where existence and uniqueness of solutions are proved and the solutions are also shown to converge to some bounded functions as $\eta\to\infty$.
We thus obtain $(\bar{T}_k,\bar{\psi}_k)$ by 
\begin{align}\label{eq:barkdef}
    \bar{T}_k(\eta,x') = \chi(\varepsilon\eta)(\tilde{T}_k(\eta,x') - \tilde{T}_{k,\infty}(x')),\quad \bar{\psi}_0(\eta,x',\beta) = \chi(\eps\eta)(\tilde{\psi}_k(\eta,x',\beta)- \tilde{\psi}_{k,\infty}(x',\beta))
\end{align}
where 
\[
\tilde{T}_{k,\infty} (x')= \lim_{\eta\to\infty} \tilde{T}_k(\eta,x'),\quad \tilde{\psi}_{k,\infty}(x',\beta) =\lim_{\eta\to\infty} \tilde{\psi}_k(\eta,x',\beta).
\]
The interior expansions $(T_k,\psi_k)$ are given, according to \eqref{eq:o/1-1}-\eqref{eq:o/2-1}, by solving the system 
\begin{align}
    &\Delta T_k + \frac{4\pi}{3} \Delta (4T_0^3 T_k) = \langle (\beta\cdot\nabla)^3 \psi_{k-3} \rangle - \frac{4\pi}{3} \Delta \mathcal{E}(T,k-1), \label{eq:otk}\\
    &\psi_k = -\beta\cdot\nabla \psi_{k-1} + 4T_0^3 T_k + \mathcal{E}(T,k-1) \label{eq:opsik}
\end{align}
with boundary conditions 
\begin{align}\label{eq:otkb}
    T_k(0,x') = \tilde{T}_{k,\infty}(x'),\quad T_k(1,x') = T_b(1,x'),\quad \text{for any }x'\in\mathbb{T}^2.
\end{align}

\emph{{Step 3: The composite approximate solution.}} 
With the above results, the composite approximate solution is given by
\begin{align}\label{eq:tapsia}
    T^a = \sum_{k=0}^N \eps^k (T_k + \bar{T}_k),\quad \psi^a = \sum_{k=0}^N \eps^k (\psi_k + \bar{\psi}_k).
\end{align}
From \eqref{eq:bconsis}, $(T^a,\psi^a)$ also satisfies the boundary conditions \eqref{eq:1b}-\eqref{eq:2b}.

\subsection{Error of the composite expansion}
In this section, we give the approximation errors. By the definition of $(\bar{T}_0,\bar{\psi}_0)$ of \eqref{eq:bar0def} and equations \eqref{eq:bT0}-\eqref{eq:bpsi0}, as well as the relation \eqref{eq:ot0b} and $\tilde{\psi}_{0,\infty} = \tilde{T}_{0,\infty}^4$, a direct computation gives
\begin{align}\label{eq:e00}
    E_0^0:&=\partial_\eta^2 \bar{T}_0 + \langle \bar{\psi}_0 - (\bar{T}_0+T_0(0))^4 + T_0^4(0) \rangle \nonumber\\
    & = (\tilde{T}_0-\tilde{T}_{0,\infty})\partial_\eta^2\chi + 2\partial_\eta\chi \partial_\eta \tilde{T}_0  + \chi \langle \tilde{T}_0^4-\tilde{T}_{0,\infty}^4\rangle - \langle (\chi(\tilde{T}_0-\tilde{T}_{0,\infty}) + \tilde{T}_{0,\infty})^4 - \tilde{T}_{0,\infty}^4 \rangle,
\end{align}
and 
\begin{align}\label{eq:e01}
    E_0^1:&=\mu \partial_\eta \bar{\psi}_0 + \bar{\psi}_0 - (\bar{T}_0+T_0(0))^4 + T_0^4(0) \nonumber\\
    &= \mu (\tilde{\psi}_0-\tilde{\psi}_{0,\infty}) \partial_\eta \chi  +  \chi (\tilde{T}_0^4-\tilde{T}_{0,\infty}^4) - ( (\chi(\tilde{T}_0-\tilde{T}_{0,\infty}) + \tilde{T}_{0,\infty})^4 - \tilde{T}_{0,\infty}^4 ),
\end{align}
Using the definition of $(\bar{T}_k,\bar{\psi}_k)$ in \eqref{eq:barkdef} and equations \eqref{eq:bTk}-\eqref{eq:bpsik} as well as the relation \eqref{eq:bTkb}-\eqref{eq:bpsikb} and $\tilde{\psi}_{k,\infty} = 4\tilde{T}_{0,\infty}^3\tilde{T}_{k,\infty}$, we get 
\begin{align}\label{eq:ek0}
    E_k^0:&=\partial_\eta^2 \bar{T}_k + \Delta_{x'}\bar{T}_{k-2} + \langle \bar{\psi}_k - \mathcal{C}(\bar{T}+P,k) + \mathcal{C}(P,k) \rangle \nonumber\\
    & = (\tilde{T}_{k}-\tilde{T}_{k,\infty})\partial_\eta^2 \chi + 2 \partial_\eta\chi \partial_\eta \tilde{T}_k + \chi(\langle 4\tilde{T}_0^3\tilde{T}_k - 4\tilde{T}_{0,\infty}^3\tilde{T}_{k,\infty} +   4(\tilde{T}_0^3-\tilde{T}_{0,\infty}^3)(P_k-P_k(0)) \rangle\nonumber\\
    &\quad+\langle\mathcal{E}(\bar{T}+P,k-1) - \mathcal{E}(P,k-1) \rangle)  -  \langle \mathcal{C}(\bar{T}+P,k) - \mathcal{C}(P,k)\rangle + (1-\chi\chi_0)\Delta_{x'}\bar{T}_{k-2},
\end{align}
and 
\begin{align}\label{eq:ek1}
    E_{k}^1:&=\mu \partial_\eta \bar{\psi}_k + \beta'\cdot\nabla_{x'}\bar{\psi}_{k-1} +  \bar{\psi}_k - \mathcal{C}(\bar{T}+P,k) + \mathcal{C}(P,k) \nonumber\\
    & = \mu(\tilde{\psi}_k-\tilde{\psi}_{k,\infty})\partial_\eta \chi + \chi(4\tilde{T}_0^3\tilde{T}_k - 4\tilde{T}_{0,\infty}^3\tilde{T}_{k,\infty} + 4(\tilde{T}_0^3-\tilde{T}_{0,\infty}^3)(P_k-P_k(0)) \nonumber\\
    &\quad +  \mathcal{E}(\bar{T}+P,k-1) - \mathcal{E}(P,k-1) ) -  \langle \mathcal{C}(\bar{T}+P,k) - \mathcal{C}(P,k)\rangle + (1-\chi\chi_0)\beta'\cdot \nabla_{x'}\bar{\psi}_{k-1}.
\end{align}
Combining  formulas \eqref{eq:R1p} and \eqref{eq:R2p}, we get from the above equations  
\begin{align}\label{eq:R1b}
    \mathcal{R}_1(T^a,\psi^a)&=  \eps^{N+1} \Delta T_{N-1} + \eps^{N+2} \Delta T_N   - \sum_{k=0}^N \eps^k \langle \mathcal{C}(T+\bar{T},k) - \mathcal{C}(T,k) - \mathcal{C}(\bar{T}+P,k) + \mathcal{C}(P,k)\rangle \nonumber\\
    &\quad- \sum_{k=N+1}^{4N} \eps^k\langle \mathcal{C}(T+\bar{T},k) \rangle  + \varepsilon^{N+1} \Delta_{x'}\bar{T}_{N-1} + \eps^{N+2}\Delta_{x'} \bar{T}_N+ \sum_{k=0}^N \eps^k E_k^0,
\end{align}
and 
\begin{align}\label{eq:R2b}
    \mathcal{R}_2(T^a,\psi^a) &=  \eps^{N+1} \beta\cdot \nabla \psi_{N}  - \sum_{k=0}^N \eps^k ( \mathcal{C}(T+\bar{T},k) - \mathcal{C}(T,k) - \mathcal{C}(\bar{T}+P,k) + \mathcal{C}(P,k)) \nonumber\\
    &\quad- \sum_{k=N+1}^{4N} \eps^k(\mathcal{C}(T+\bar{T},k))  + \varepsilon^{N+1} \beta'\cdot\nabla_{x'} \bar{\psi}_N + \sum_{k=0}^N \eps^k E_k^1.
\end{align}

Therefore, we summarize the above results in the following lemma.
\begin{lemma} \label{lem.ap}
    Let $(T^a, \psi^a)$ be given by \eqref{eq:tapsia}. Then $(T^a,\psi^a)$ satisfies the equations \eqref{eq:1a}-\eqref{eq:2a}
    with boundary conditions \eqref{eq:1ab}-\eqref{eq:2ab},
    where $\mathcal{R}_1(T^a,\psi^a)$, $\mathcal{R}_2(T^a,\psi^a)$ are given by \eqref{eq:R1b} and \eqref{eq:R2b}, respectively. 
\end{lemma}

\section{Properties of the approximate solutions}\label{sec3}
In this section, we study the properties of the approximate solution $(T^a,\psi^a)$ constructed in the previous section. The properties of the interior approximations are shown first, followed by the properties of the boundary layer approximations. Then, the approximate errors $\mathcal{R}_1(T^a,\psi^a)$ and $\mathcal{R}_2(T^a,\psi^a)$, obtained in the previous section, are computed in Theorem \ref{thm-ap}. Finally, the coercivity inequality \eqref{eq:est/gg} is shown in Lemma \ref{lm.g2}.

\subsection{Interior approximations}
The interior approximation $(T_0,\psi_0)$ is obtained by solving \eqref{eq:ot0} and $(T_k,\psi_k)$, $k=1,\ldots,N$ are obtained by solving \eqref{eq:otk}. For equation \eqref{eq:ot0}, the following lemma holds.
\begin{lemma}\label{lem:o0}
 Assume $\tilde{T}_{0,\infty} \in L^2(\mathbb{T}^2)$ and $T_{b}(1,x') \in L^2(\mathbb{T}^2)$ satisfy $\tilde{T}_{0,\infty}\ge 0$, $T_b(1,x')\ge 0$ for any $x'\in \mathbb{T}^2$. Then, there exists a unique solution $T_0\in C^\infty(\Omega)$ to equation \eqref{eq:ot0} with boundary condition \eqref{eq:ot0b} and the solution satisfies $T_0(x) \ge 0$ for any $x\in\Omega$.
\end{lemma}
\begin{proof}
    The proof of the above lemma follows directly from elliptic theory. Actually, we may take $u=T_0+4\pi T_0^4/3$ and then $\Delta u=0$ in $\Omega$ and $u(0,x')=\tilde{T}_{0,\infty}(x') + 4\pi \tilde{T}^4_{0,\infty}(x')/3$, $u(1,x')=T_b(1,x')+4\pi T_b^4(1,x')/3$ on the boundary. We have $u \in C^\infty(\Omega)$ and so $T\in C^\infty(\Omega)$.
\end{proof}
We now provide  in the following lemma a well-posedness result for the elliptic equation \eqref{eq:otk}.
\begin{lemma}\label{lem:ok}
Assume $\tilde{T}_{k,\infty} \in L^2(\mathbb{T}^2)$ satisfy $\tilde{T}_{k,\infty}\ge 0$. Given $T_s,1\le s\le k-1$ and $\psi_s,1\le s\le k-1$ satisfy $T_s \in L^2(\Omega)$, $\psi_s\in L^2(\Omega\times\mathbb{S}^2).$ Then there exists a unique solution $T_k \in C^\infty(\Omega)$ to equation \eqref{eq:otk} with boundary conditions \eqref{eq:otkb}.
\end{lemma}
\begin{proof}
    The proof of the above lemma also follows directly from elliptic theory by taking $u=(1+16\pi T_0^3/3)T_k$.
\end{proof}
We finish this part by giving an $L^{\infty}$-estimate of the Taylor expansion  $P_k$ defined by \eqref{eq:Pkdef}.
\begin{lemma}\label{lem:Pdecay}
    Let $T_s$, $s=0,\ldots,N$ be the solution to equation \eqref{eq:ot0} and \eqref{eq:otk}, and $P_s$, $s=0,\ldots,N$ be given by \eqref{eq:Pkdef}. Then  
    \begin{align}
        \|P_{s}(\eta)\|_{L^\infty(\mathbb{T}^2)} \le C(1+\eta^s),\quad \text{for any }s=0,\ldots,N,
    \end{align}
    for some constant $C>0$.
\end{lemma}
\begin{proof}
    Since $T_s \in C^\infty(\Omega)$, $\partial_{x_1}^lT_s(0)$ are bounded for any $l\ge 0$. Therefore, 
    \begin{align}
        |P_s| &= \left| \sum_{l=0}^s\frac{\eta^l}{l!}\frac{\partial^l}{\partial {x}_1^l} T_{s-l}(0,x')\right| 
        \le C_1\sum_{l=0}^s \frac{\eta^l}{l!} \le C (1+\eta^s),
    \end{align}
    and thus the lemma holds.
\end{proof}
\subsection{Boundary layer approximation}\label{subsecBLA32} 

We now turn our attention to the half-space nonlinear and linear Milne problems \eqref{eq:bT0}-\eqref{eq:bpsi0} and \eqref{eq:bTk}-\eqref{eq:bpsik}. Their analyses have been carried out in \cite{Bounadrylayer2019GHM2}. Here we summarize the results in Lemma \ref{lem:bd0} and \ref{lem:bdk}, for the proof we refer the reader to \cite{Bounadrylayer2019GHM2}. First, we have the following lemma for system \eqref{eq:bT0}-\eqref{eq:bpsi0}.
\begin{lemma}[\cite{Bounadrylayer2019GHM2}]\label{lem:bd0}
    Given $(T_b(0,x'),\psi_b(0,x',\beta)) \in L^\infty(\mathbb{T}^2)\times L^\infty(\Gamma_-)$ satisfying $T_b(0,x')\ge 0$, $\psi_b(0,x',\beta)\ge 0$ for any $x'\in\mathbb{T}^2$ and $\beta\in\mathbb{S}^2$ with $\beta_1\ge 0$.    
    There exists a bounded solution $(\tilde{T}_0,\tilde{\psi}_0) \in L^2_{\rm loc}(\mathbb{R}_+\times \mathbb{T}^2)\times L^2_{\rm loc}(\mathbb{R}_+\times \mathbb{T}^2 \times \mathbb{S}^2)$ to system \eqref{eq:bT0}-\eqref{eq:bpsi0} with boundary conditions \eqref{eq:bT0b}-\eqref{eq:bpsi0b}. Moreover, there exists a constant function $\tilde{T}_{0,\infty} \in L^\infty(\mathbb{T}^2)$ such that 
    \begin{align}\label{eq:b0decay}
        |\tilde{T}_0(\eta,x') - \tilde{T}_{0,\infty}(x')| \le C e^{-\lambda_0 \eta}, \quad |\tilde{\psi}(\eta,x',\beta) - \tilde{T}_{0,\infty}^4(x') | \le C e^{-\lambda_0 \eta},
    \end{align}
    for any $\eta \in [0,\infty)$ and $\beta \in \mathbb{S}^2$. Here $C>0$ are constants depending linearly on the constant $\left(\int_0^1 \mu (\psi_b-T_b^4)^2 d\mu \right)^{\frac12}$ and $\lambda_0$ is any fixed constant in $[0,1)$.

    Furthermore, for sufficiently small value of $\left(\int_0^1 \mu (\psi_b-T_b^4)^2 d\mu \right)^{\frac12}$, the solution $(\tilde{T}_0,\tilde{\psi}_0)$ is unique.
\end{lemma}
Note that by the above lemma, $\tilde{\psi}_{0,\infty} : = \lim_{\eta\to 0} \tilde{\psi}_0(\eta) = \tilde{T}_{0,\infty}^4$.

For system \eqref{eq:bTk}-\eqref{eq:bpsik}, the spectral assumption \ref{asA} on $\tilde{T}_0$ is needed in order to show the existence, uniqueness and exponential decay of  solutions. The following lemma holds.
\begin{lemma}[\cite{Bounadrylayer2019GHM2}]\label{lem:bdk}
    Given $\bar{T}_s \in L^2(\mathbb{R}_+;C^2(\mathbb{T}^2))$, $\bar{\psi}_s \in L^2(\mathbb{R}_+;C^2(\mathbb{T}^2\times\mathbb{S}^2))$, $0\le s\le k-1$ such that 
    \begin{align} \label{eq:lm4decay}
        |\bar{T}_s(\eta,x')|\le C e^{-\lambda_{k-1} \eta},\quad |\bar{\psi}_s(\eta,x',\beta)|\le Ce^{-\lambda_{k-1} \eta}
    \end{align}
for some constants $C>0$ and for any $x'\in\mathbb{T}^2$, $\beta\in\mathbb{S}^2$, $\lambda_{k-1}>0$. Given $P_s\in C([0,\infty)\times \mathbb{T}^2)$ satisfying $|P_s|\le C(1+\eta^{k-1})$ for some constant $C>0$ for any $0\le s\le k-1$. Assume $\tilde{T}_0$ satisfies the spectral assumption \ref{asA}. Then there exists a unique bounded solution $(\tilde{T}_k,\tilde{\psi}_k) \in L^2_{\rm loc}(\mathbb{R}_+\times\mathbb{T}^2)\times L^2_{\rm loc}(\mathbb{R}_+\times \mathbb{T}^2\times \mathbb{S}^2)$ to system \eqref{eq:bTk}-\eqref{eq:bpsik} with boundary conditions \eqref{eq:bTkb}-\eqref{eq:bpsikb}. Moreover, there exists constants $(\tilde{T}_{k,\infty}, \tilde{\psi}_{k,\infty}) \in L^\infty(\mathbb{T}^2)\times L^\infty(\mathbb{T}^2\times\mathbb{S}^2)$ such that 
    \begin{align}\label{eq:bkdecay}
        |\tilde{T}_k(\eta,x') - \tilde{T}_{k,\infty}(x')| \le C e^{-\lambda' \eta}, \quad |\tilde{\psi}(\eta,x',\beta) -\tilde{\psi}_{k,\infty}(x',\beta)| \le C e^{-\lambda' \eta},
    \end{align}
    for any constant $0<\lambda'<\lambda_{k-1}$,
 where $C>0$ is a positive constant independent of $k$ and the relation $\tilde{\psi}_{k,\infty} = 4\tilde{T}^3_{0,\infty}\tilde{T}_{k,\infty}$ holds.
    \end{lemma}
\begin{proof}   
    It has been proved in \cite{Bounadrylayer2019GHM2} that given functions $S_1=S_1(\eta,x')$, $S_2=S_2(\eta,x',\beta)$ such that 
    \begin{align}\label{eq:Sbd}
        \int_0^\infty e^{2\lambda' \eta}S_1^2 d\eta, \, \int_0^\infty \int_{\mathbb{S}^2} e^{2\lambda' \eta} S_2^2 d\eta d\beta < \infty
    \end{align}
    are bounded, the following equations 
    \begin{align}
        &\partial_\eta^2 g + \langle \phi - 4\tilde{T}_0^3 g \rangle = S_1,\label{eq:gtmp}\\
        &\mu \partial_\eta \phi + \phi - 4\tilde{T}_0^3 g = S_2,\label{eq:phitmp}
    \end{align}
    with boundary conditions 
    \begin{align*}
        &g(\eta=0,x') = 0,\quad \text{for any }x'\in \mathbb{T}^2,\\
        &\phi(\eta=0,x',\beta) = \phi_b,\quad \text{for any } (x,\beta) \in \Gamma_-,
    \end{align*}
    have a unique bounded solution $(g,\phi) \in L^2_{\rm loc}(\mathbb{R}_+\times\mathbb{T}^2)\times L^2_{\rm loc}(\mathbb{R}_+\times \mathbb{T}^2\times \mathbb{S}^2)$. Moreover, there exists a constant $g_\infty$ such that $|g - g_\infty| \le C e^{-\lambda' \eta}$, $|\phi - 4\tilde{T}_{0,\infty}^3 g_\infty| \le C e^{-\lambda' \eta}$, for the proof of this result we refer to the \cite[Theorem 2]{Bounadrylayer2019GHM2}. Therefore, we deduce that  $\tilde{\psi}_{k,\infty} = 4\tilde{T}^3_{0,\infty}\tilde{T}_{k,\infty}$.
    
    Taking $g=\tilde{T}_k$, $\phi=\tilde{\psi}_k$ and $S_1 = -\chi_0 \Delta_{x'} \bar{T}_{k-2} + \langle 4(\tilde{T}_0^3-\tilde{T}_{0,\infty}^3)(P_k-P_k(0)) \rangle+\langle \mathcal{E}(\bar{T}+P,k-1) - \mathcal{E}(P,k-1) \rangle$, $S_2 =-\chi_0\beta'\cdot\nabla_{x'} \bar{\psi}_{k-1} + 4(\tilde{T}_0^3-\tilde{T}_{0,\infty}^3)(P_k-P_k(0))  + (\mathcal{E}(\bar{T}+P,k-1) - \mathcal{E}(P,k-1))$, $\phi_b= \beta\cdot\nabla \psi_{k-1}(0) - \mathcal{E}(T(0),k-1)$, system \eqref{eq:gtmp}-\eqref{eq:phitmp} becomes system \eqref{eq:bTk}-\eqref{eq:bpsik} with boundary conditions \eqref{eq:bTkb}-\eqref{eq:bpsikb}. Therefore, Lemma \ref{lem:bdk} holds if $S_1$, $S_2$ satisfy \eqref{eq:Sbd}.

    To show \eqref{eq:Sbd}, from the assumption \eqref{eq:lm4decay},
    \begin{align}
        \int_0^\infty e^{2\lambda'\eta} |\chi_0\Delta_{x'}\bar{T}_{k-2}|^2 d\eta \le C \int_0^\infty e^{2\lambda'\eta} e^{-2\lambda_{k-1}\eta} d\eta = \frac{C}{2(\lambda_{k-1}-\lambda')}, \label{eq:lm4/1}\\
        \int_0^\infty e^{2\lambda'\eta} |-\chi_0\beta'\cdot\nabla_{x'}\bar{\psi}_{k-1}|^2 d\eta \le  C \int_0^\infty e^{2\lambda'\eta} e^{-2\lambda_{k-1}\eta} d\eta = \frac{C}{2(\lambda_{k-1}-\lambda')}.\label{eq:lm4/2}
    \end{align}
    Moreover, due to the assumption \eqref{eq:lm4decay} and $|P_s|\le C(1+\eta^k)$, we have 
    \begin{align*}
        | 4(\tilde{T}_0^3-\tilde{T}_{0,\infty}^3)(P_k-P_k(0))| &= |4(\tilde{T}_0-\tilde{T}_{0,\infty}) (P_k-P_{k}(0)| |(\tilde{T}_0^2 + \tilde{T}_{0,\infty}^2)| \le C e^{-\lambda_0 \eta} (1+\eta^k),
    \end{align*}
    and
    \begin{align}
        &|\mathcal{E}(\bar{T}+P,k-1) - \mathcal{E}(P,k-1)| \nonumber\\
        &\quad = \Big| \sum_{\substack{i+j+l+m=k\\i,j,l,m\ge 1}} (\bar{T}_i+P_i) (\bar{T}_j+P_j)(\bar{T}_l+P_l)(\bar{T}_m+P_m) - \sum_{\substack{i+j+l+m=k\\i,j,l,m\ge 1}} P_iP_jP_lP_m\Big| \nonumber\\
        &\quad =\Big|\sum_{\substack{i+j+l+m=k\\i,j,l,m\ge 1}} (\bar{T}_i\bar{T}_j\bar{T}_l\bar{T}_m+3\bar{T}_i P_jP_lP_m + 6\bar{T}_i\bar{T}_jP_lP_m + 3\bar{T}_i\bar{T}_j\bar{T}_lP_m) \Big| \nonumber\\
        &\quad \le C (e^{-4\lambda_{k-1} \eta} + e^{-\lambda_{k-1} \eta} (1+\eta^{k-1})^3 + e^{-2\lambda_{k-1}\eta}(1+\eta^{k-1})^2+e^{-3\lambda_{k-1} \eta}(1+\eta^{k-1})) \nonumber\\
        &\quad \le C e^{-\lambda_{k-1} \eta}(1+\eta^{3k-3}).
    \end{align}
    Therefore, 
    \begin{align}
        \int_0^\infty e^{2\lambda'\eta} | 4(\tilde{T}_0^3-\tilde{T}_{0,\infty}^3)(P_k-P_k(0)) |^2 d\eta &\le C \int_0^\infty e^{2\lambda'\eta} e^{-2\lambda_0\eta}(1+\eta^k) d\eta \nonumber\\
        & = \frac{C}{2(\lambda_0-\lambda')} + \frac{C}{2^{k+1}(\lambda_0-\lambda')^{k+1}}\Gamma(k+1),
    \end{align}
    and
    \begin{align}
        \int_0^\infty e^{2\lambda' \eta} |\mathcal{E}(\bar{T}+P,k-1) - \mathcal{E}(P,k-1)|^2  d\eta &\le C \int_0^\infty e^{2\lambda'\eta} e^{-2\lambda_{k-1} \eta} (1+\eta^{6k-6}) d\eta \nonumber\\
        &= \frac{C}{2(\lambda_{k-1}-\lambda')} + \frac{C}{2^{6k-5}(\lambda_{k-1}-\lambda')^{6k-5}} \Gamma(6k-5),
    \end{align}
    where $\Gamma(n):=(n-1)!$ is the Gamma function. Combining the above inequalities with \eqref{eq:lm4/1}-\eqref{eq:lm4/2} implies \eqref{eq:Sbd} and finishes the proof.
\end{proof}
In the above lemma, the constant $\lambda'$ may vary for different $k$. In order to get a uniform constant, we apply the above lemma iteratively. First, according to \eqref{eq:b0decay} from Lemma \ref{lem:bd0}, $\bar{T}_0=\chi(\tilde{T}_0 - \tilde{T}_{0,\infty})$  and $\bar{\psi}_0=\chi(\tilde{\psi}_0-\tilde{\psi}_{0,\infty})$ satisfy 
\begin{align}
    |\bar{T}_0(\eta,x')| \le C e^{-\lambda_0 \eta}, \quad |\bar{\psi}_0(\eta,x',\beta| \le C e^{-\lambda_0 \eta},\quad \text{for any }x'\in \mathbb{T}^2, \,\beta\in \mathbb{S}^2.
\end{align}
for some $C>0$ and $\lambda>0$. Therefore, with the definition of $P_0$ in \eqref{eq:Pkdef}, we have that $P_0=T_0(0)$ is bounded, i.e. $|P_0|\le C$ for some constant $C>0$. Therefore, the assumptions of Lemma \ref{lem:bdk} with $k=1$ hold. We take $\lambda'=\lambda_0-\eps_0$, with $\eps_0>0$ being a sufficiently small number, in \eqref{eq:bkdecay} and get that for any $x'\in \mathbb{T}^2,\,\beta\in \mathbb{S}^2$,
\begin{align}
    |\tilde{T}_1 (\eta,x')- \tilde{T}_{1,\infty}(x')|\le Ce^{-(\lambda_0-\eps_0)\eta},\quad |\tilde{\psi}_1(\eta,x',\beta) - \tilde{\psi}_{1,\infty}(x',\beta)| \le C e^{-(\lambda_0-\eps_0)\eta}.
\end{align}
With these, $\bar{T}_1=\chi(\tilde{T}_1-\tilde{T}_{1,\infty})$ and $\bar{\psi}_1=\chi(\tilde{\psi}_1-\tilde{\psi}_{1,\infty})$ satisfy 
\begin{align}
    |\bar{T}_1(\eta,x')|\le Ce^{-(\lambda_0-\eps_0)\eta},\quad |\bar{\psi}_k(\eta,x',\beta)| \le Ce^{-(\lambda_0-\eps_0)\eta},\quad \forall ~ (x'\in \mathbb{T}^2,\,\beta\in \mathbb{S}^2).
\end{align}
Moreover, from Lemma \ref{lem:Pdecay}, $|P_1| \le C(1+\eta)$.
Therefore, the assumption of Lemma \ref{lem:bdk} with $k=2$ hold with $\lambda_1=\lambda_0-\eps_0$. We take $\lambda'=\lambda_1-\eps_0/4$ in \eqref{eq:bkdecay} and get that for any $x'\in \mathbb{T}^2,\,\beta\in \mathbb{S}^2$,
\begin{align}
    |\tilde{T}_2 (\eta,x')- \tilde{T}_{2,\infty}(x')|\le Ce^{-(\lambda_1-\eps_0/4)\eta},~ |\tilde{\psi}_2(\eta,x',\beta) - \tilde{\psi}_{2,\infty}(x',\beta)| \le C e^{-(\lambda_1-\eps_0/4)\eta}.
\end{align}
These again implies $|\bar{T}_2(\eta,x')|,|\bar{\psi}_2(\eta,x',\beta)|\le Ce^{-(\lambda_2-\eps_0/8)}$ with $\lambda_2=\lambda_1-\eps_0/4$ if we take $\lambda'=\lambda_2-\eps_0/9$. We can thus apply Lemma \ref{lem:bdk} iteratively with $\lambda'=\lambda_{s-1} - \eps_0/s^2$ and $\lambda_s=\lambda_{s-1} - \eps_0/(s-1)^2$ in the $s$-th step and get that 
\begin{align}
    |\tilde{T}_s (\eta,x')- \tilde{T}_{s,\infty}(x')|\le Ce^{-(\lambda_{s-1}-\eps_0/s^2)\eta},~ |\tilde{\psi}_k(\eta,x',\beta) - \tilde{\psi}_{2,\infty}(x',\beta)| \le C e^{-(\lambda_s-\eps_0/s^2)\eta}, 
\end{align}
hold for any $0\le s \le N$ for all $x'\in\mathbb{T}^2$, $\beta\in \mathbb{S}^2$. By $\bar{T}_s=\chi(\tilde{T}_s-\tilde{T}_{s,\infty})$ and $\bar{\psi}_s=\chi(\tilde{\psi}_s-\tilde{\psi}_{s,\infty})$, the above inequalities imply for any $s=0,\ldots,N$,
\begin{align}
    |\bar{T}_s(\eta,x')| \le Ce^{-(\lambda_{s-1}-\eps_0/s^2)\eta},\quad |\bar{\psi}_s(\eta,x')| \le Ce^{-(\lambda_{s-1}-\eps_0/s^2)\eta},\quad  \forall ~ (x'\in \mathbb{T}^2,\,\beta\in \mathbb{S}^2).
\end{align}
Since $\lambda_{s}=\lambda_{s-1}-\eps_0/(s-1)^2$, 
\begin{align}
    \lambda_{N-1} - \eps_0/N^2 = \lambda_{N-2}-\eps_0/(N-1)^2  - \eps_0/N^2 = \cdots = \lambda_0 - \eps_0\left(1+\frac{1}{2^2}+\cdots+\frac{1}{N^2}\right) .
\end{align}
From the formula $\sum_{n=1}^\infty \frac{1}{n^2} = \pi^2/6$, the above equation implies 
\begin{align}
   \lambda_{N-1} - \eps_0/N^2  = \lambda_0 - \eps_0\sum_{n=1}^N \frac{1}{N^2} \le \lambda_0 - \eps_0\sum_{n=1}^\infty \frac{1}{n^2} \le \lambda_0 - \frac{\pi^2}{6} \eps_0.
\end{align}
Therefore, taking $\eps_0=3\lambda_0/\pi^2$ and setting $\lambda= \lambda_0/2$, the following lemma holds. 
\begin{lemma}\label{lem:bdecay}
    Let $\Big\{\left(\tilde{T}_s,\tilde{\psi}_s\right)\Big\}_{0\le s \le N}$ be solutions to system \eqref{eq:bT0}-\eqref{eq:bpsi0} for $s=0$ and \eqref{eq:bTk}-\eqref{eq:bpsik} for $s\ge 1$. Then there exist constants $\lambda>0$ and $C>0$ such that for any $s=0,\ldots,N$,
    \begin{equation}\label{eq:bdecay}
    \begin{aligned}
   &|\tilde{T}_s(\eta,x') - \tilde{T}_{s,\infty}(x')|,\,|\tilde{\psi}_s(\eta,x',\beta)-\tilde{\psi}_{s,\infty}(x',\beta)| \le C e^{-\lambda\eta},\\
    &|\bar{T}_s(\eta,x')|, \, |\bar{\psi}_s(\eta,x',\beta)| \le C e^{-\lambda\eta},\quad \text{for any } (x',\beta)\in\mathbb{T}^2\times\mathbb{S}^2,
\end{aligned}
\end{equation}
where $C>0$ is a postive constant independent of $s$.
\end{lemma}

\subsection{Residual Estimates}
Next  we estimate $\mathcal{R}_1(T^a,\psi^a)$ and $\mathcal{R}_2(T^a,\psi^a)$ given by \eqref{eq:R1b} and \eqref{eq:R2b}, respectively.
We will prove the following theorem.
\begin{theorem}\label{thm-ap}
   Assume $\tilde{T}_0$ satisfies the spectral assumption \ref{asA}. The composite approximate solution $(T^a,\psi^a)$ constructed in section \ref{sec2}, satisfies \eqref{eq:1a}-\eqref{eq:2a} with boundary conditions \eqref{eq:1ab}-\eqref{eq:2ab}. Moreover, the error terms $\mathcal{R}_1=\mathcal{R}_1(T^a,\psi^a), \mathcal{R}_2= \mathcal{R}_2(T^a,\psi^a)$ satisfy 
    \begin{align}
        &\|\mathcal{R}_1(T^a,\psi^a)\|_{L^2(\Omega)}, ~\|\mathcal{R}_1(T^a,\psi^a)\|_{L^\infty(\Omega)},~
        \|\mathcal{R}_2(T^a,\psi^a)\|_{L^2(\Omega\times\mathbb{S}^2)}, ~\|\mathcal{R}_2(T^a,\psi^a)\|_{L^\infty(\Omega\times\mathbb{S}^2)} \nonumber\\
        &\quad\le        C\gamma_N \eps^{N+1} + (N+2) Ce^{-\frac{\lambda \delta}{4\eps}},
    \end{align}
    for some constant $\gamma_N>1$. Furthermore, for $\delta > -\frac{4}{\lambda}(N+1)\eps \log \eps $, the above estimate implies 
    \begin{align}\label{eq:errest}
        &\|\mathcal{R}_1(T^a,\psi^a)\|_{L^2(\Omega)}, ~\|\mathcal{R}_1(T^a,\psi^a)\|_{L^\infty(\Omega)},~
        \|\mathcal{R}_2(T^a,\psi^a)\|_{L^2(\Omega\times\mathbb{S}^2)}, ~\|\mathcal{R}_2(T^a,\psi^a)\|_{L^\infty(\Omega\times\mathbb{S}^2)} \nonumber\\
        &\quad\le C\eps^{N+1},
    \end{align}
    where $C>0$ is a positive constant independent of $\eps$.
\end{theorem}

\begin{proof}
    We first consider $\mathcal{R}_1(T^a,\psi^a)$. From \eqref{eq:R1b}, we have
    \begin{align}
        \mathcal{R}_1(T^a,\psi^a) &= (\eps^{N+1} \Delta T_{N-1} + \eps^{N+2} \Delta T_N + \varepsilon^{N+1} \Delta_{x'}\bar{T}_{N-1} + \eps^{N+2}\Delta_{x'} \bar{T}_N) \nonumber\\
        &\quad- \sum_{k=N+1}^{4N} \eps^k\langle \mathcal{C}(T+\bar{T},k) \rangle - \sum_{k=0}^N \eps^k \langle \mathcal{C}(T+\bar{T},k) - \mathcal{C}(T,k) - \mathcal{C}(\bar{T}+P,k) + \mathcal{C}(P,k)\rangle \nonumber\\
        &\quad + \sum_{k=0}^N \eps^k E_k^0 \nonumber\\
        &=: R_{11} + R_{12} + R_{13} + R_{14}.  
    \end{align}
    By Lemma \ref{lem:ok}, $\|\Delta T_{N-1}\|_{L^2(\Omega)}$, $\|\Delta T_{N}\|_{L^2(\Omega)}$, $\|\Delta T_{N-1}\|_{L^\infty(\Omega)}$ and $\|\Delta T_{N}\|_{L^\infty(\Omega)}$ are bounded. Moreover, from Lemma \ref{lem:bdecay},
    \begin{align*}
        \int_0^1 \bar{T}_{N-1}^2 dx_1 &= \int_0^1 \chi^2(x_1)(\tilde{T}_{N-1}(x_1/\eps)-\tilde{T}_{N-1,\infty})^2 dx_1 = \int_0^{1/\eps} \chi^2(\eps\eta)(\tilde{T}_{N-1}(\eta)-\tilde{T}_{N-1,\infty})^2 d\eta \nonumber\\
        &\le \int_0^\infty (\tilde{T}_{N-1}-\tilde{T}_{N-1,\infty})^2 d\eta \le C \int_0^\infty e^{-2\lambda \eta} d\eta = \frac{C}{2\lambda},\nonumber\\
        \int_0^1 \bar{\psi}_{N-1}^2 dx_1 &= \int_0^1 \chi^2(x_1)(\tilde{\psi}_{N-1}(x_1/\eps)-\tilde{\psi}_{N-1,\infty})^2 d\eta \le C \int_0^\infty e^{-2\lambda \eta} d\eta = \frac{C}{2\lambda},
    \end{align*}
    The above inequalities also hold for $(\bar{T}_N,\bar{\psi}_N)$. 
    It follows from the above inequalities and \eqref{eq:bdecay} that  
    $\|\Delta_{x'} \bar{T}_{N-1}\|_{L^2(\Omega)}$, $\|\Delta_{x'} \bar{T}_{N}\|_{L^2(\Omega)}$, $\|\Delta_{x'} \bar{T}_{N-1}\|_{L^\infty(\Omega)}$ and $\|\Delta_{x'} \bar{T}_{N}\|_{L^\infty(\Omega)}$ are bounded. Therefore, 
    \begin{align}\label{eq:R11est}
        \|R_{11}\|_{L^2(\Omega)},~\|R_{11}\|_{L^\infty(\Omega)} \le C\eps^{N+1}.
    \end{align}
    Due to \eqref{eq:bdecay}, $\|\bar{T}_s\|_{L^2(\Omega)}$, $\|\bar{T}_s\|_{L^\infty(\Omega)}$ are bounded for any $s=0,\ldots,N$. Therefore, 
    \begin{align}\label{eq:R12est}
        \|R_{12}\|_{L^2(\Omega)},~\|R_{12}\|_{L^\infty(\Omega)} \le C \eps^{N+1}.
    \end{align}
    Next we consider $R_{13}$. Since $\chi(x_1)=0$ for $x_1\ge \tfrac38\delta$, $R_{13}(x_1,x')=0$ for $x_1\ge \tfrac38\delta$. For the region $x_1\le \tfrac38 \delta$, we rewrite it as 
    \begin{align}
        R_{13} 1_{\{x_1\le \tfrac38\delta\}}& = -\sum_{k=0}^{4N}\eps^k   \langle \mathcal{C}(T+\bar{T},k) - \mathcal{C}(T,k) - \mathcal{C}(\bar{T}+P,k) + \mathcal{C}(P,k)\rangle 1_{\{x_1\le \tfrac38\delta\}} \nonumber\\
        &\quad + \sum_{k=N+1}^{4N}\eps^k \langle \mathcal{C}(T+\bar{T},k) - \mathcal{C}(T,k) - \mathcal{C}(\bar{T}+P,k) + \mathcal{C}(P,k)\rangle 1_{\{x_1\le \tfrac38\delta\}}\nonumber \\
        &=:R_{131} + R_{132}.
    \end{align}
    Similarly as $R_{12}$, $R_{132}$ satisfies 
    \begin{align}\label{eq:R132}
        \|R_{132}\|_{L^2(\Omega)},~\|R_{132}\|_{L^\infty(\Omega)} \le C \eps^{N+1}.
    \end{align}
    For $R_{132}$, from the definition \eqref{eq:mathcalC} of $\mathcal{C}$,
    \begin{align}
        R_{131} &=  -\left\langle  \left(\sum_{k=0}^N \eps^k (T_k+\bar{T}_k)\right)^4 - \left(\sum_{k=0}^N \eps^k T_k\right)^4  - \left(\sum_{k=0}^N \eps^k (P_k+\bar{T}_k)\right)^4 + \left(\sum_{k=0}^N \eps^k P_k\right)^4 \right\rangle.
    \end{align}
    Using the formula $a^4-b^4=(a-b)(a+b)(a^2+b^2)$ and for $a-b=c-d=f$
    \begin{align*}
        (a^4-b^4) - (c^4 - d^4) &=  (a-b)(a+b)(a^2+b^2) - (c-d)(c+d)(c^2+d^2) \nonumber\\
        &=f(a-c)\left(2a^2+2b^2+(c+d)(a+b+c+d)\right),
    \end{align*} 
    with $a=\sum_{k=0}^N\eps^k(T_k+\bar{T}_k)$, $b=\sum_{k=0}^N \eps^k T_k$, $c=\sum_{k=0}^N \eps^k(P_k+\bar{T}_k)$, $d=\sum_{k=0}^N \eps^k P_k$, $f=\sum_{k=0}^N \eps^k \bar{T}_k$, we obtain 
    \begin{align}\label{eq:ttpp}
        &\left(\sum_{k=0}^N \eps^k (T_k+\bar{T}_k)\right)^4 - \left(\sum_{k=0}^N \eps^k T_k\right)^4  - \left(\sum_{k=0}^N \eps^k (P_k+\bar{T}_k)\right)^4 + \left(\sum_{k=0}^N \eps^k P_k\right)^4 \nonumber\\
        &\quad = \left(\sum_{k=0}^N \eps^k \bar{T}_k\right)\left(\sum_{k=0}^N \eps^k (T_k-P_k)\right) \left(2a^2+2b^2+(c+d)(a+b+c+d)\right).
    \end{align}
    Due to \eqref{eq:bdecay}, 
    \begin{align}\label{eq:Tabsdecay}
        \left|\sum_{k=0}^N \eps^k \bar{T}_k\right| \le C e^{-\lambda\eta} = Ce^{-\frac{\lambda x_1}{\eps}}.
    \end{align}
    Taylor's formula yields 
    \begin{align}
        T_k(x_1,x') = \sum_{l=0}^{N-k} \frac{x_1^l}{l!} \partial_{x_1}^l T_l(0,x') + \frac{\partial_{x_1}^{N-k+1}}{(N-k+1)!} T_k(\xi_k,x') x_1^{N-k+1}.
    \end{align}
    Using the above formula, we get 
    \begin{align}\label{eq:T-P}
        \sum_{k=0}^N \eps^k(T_k - P_k) &= \sum_{k=0}^N \eps^k T_k - \sum_{k=0}^N \eps^k P_k \nonumber\\
        &= \sum_{k=0}^N \eps^k \left(\sum_{l=0}^{N-k}\frac{x_1^l}{l!} \partial_{x_1}^l T_k(0,x') + \frac{\partial_{x_1}^{N-k+1} T_k(\xi_k,x')}{(N-k+1)!} x_1^{N-k+1}\right) \nonumber\\
        &\quad - \sum_{k=0}^N \eps^k \sum_{l=0}^k \frac{\eta^l}{l!} \frac{\partial^l}{\partial x_1^l} T_{k-l}(0,x'),
    \end{align}
    with $\xi_k\in [0,x_1]$.
    Using the formula 
    $$ \sum_{k=0}^N \sum_{l=0}^{k} f(l,k)=\sum_{l=0}^N\sum_{k=l}^{N} f(l,k) \stackrel{s=k-l}{=\joinrel=} \sum_{l=0}^N\sum_{s=0}^{N-l} f(l,l+s) \stackrel{l\to k, s\to l}{=\joinrel=} \sum_{k=0}^N\sum_{l=0}^{N-k}f(k,k+l)$$
    and taking $f(k,k+l) = \eps^k\frac{x_1^l}{l!}\partial_{x_1}^l T_k(0,x')$, we get 
    $$f(l,k)=\eps^l \frac{x_1^{k-l}}{(k-l)!}\partial_{x_1}^{k-l}T_l(0,x')= \eps^k \frac{\eta^{k-l}}{(k-l)!} \partial_{x_1}^{k-l}T_l(0,x')$$
    and so 
    \begin{align*}
        \sum_{k=0}^N \sum_{l=0}^{N-k}\eps^k\frac{x_1^l}{l!}\partial_{x_1}^l T_k(0,x')= \sum_{k=0}^N \sum_{l=0}^{k} \eps^k \frac{\eta^{k-l}}{(k-l)!} \partial_{x_1}^{k-l}T_l(0,x') \stackrel{k\to k,k-l \to l}{=\joinrel=} \sum_{k=0}^N \sum_{l=0}^N \eps^k \frac{\eta^l}{l!} \partial_{x_1}^l T_{k-l}(0,x').
    \end{align*}
    Taking this relation into \eqref{eq:T-P} leads to 
    \begin{align*}
        \sum_{k=0}^N \eps^k(T_k-P_k) = \sum_{k=0}^N \eps^k \frac{\partial_{x_1}^{N-k+1} T_k(\xi_k,x')}{(N-k+1)!} x_1^{N-k+1}.
    \end{align*}
  Combining this with \eqref{eq:Tabsdecay}, \eqref{eq:ttpp} satisfies 
    \begin{align*}
        &\left|\left(\sum_{k=0}^N \eps^k (T_k+\bar{T}_k)\right)^4 - \left(\sum_{k=0}^N \eps^k T_k\right)^4  - \left(\sum_{k=0}^N \eps^k (P_k+\bar{T}_k)\right)^4 + \left(\sum_{k=0}^N \eps^k P_k\right)^4 \right| \nonumber\\
        &\quad \le C e^{-\frac{\lambda x_1}{\eps}} \sum_{k=0}^N \eps^k \frac{|\partial_{x_1}^{N-k+1} T_k(\xi_k,x')|}{(N-k+1)!} x_1^{N-k+1} |2a^2+2b^2+(c+d)(a+b+c+d)| \nonumber\\
        &\quad \le C\sum_{k=0}^N \eps^k \frac{1}{(N-k+1)!}x_1^{N-k+1} e^{-\frac{\lambda x_1}{\eps}},
    \end{align*}
    where the fact that $\|a,b,c,d\|_{L^\infty(\Omega)}\le C$ are bounded and $\|T_k\|_{C^s(\Omega)}\le C$ is bounded for any $s\ge 0$. Note that the function $h(x_1):=x_1^{N-k+1} e^{-\lambda x_1/\eps}$ attains its maximum at $x_1^*= (N-k+1)\eps/\lambda$ with the maximum value $h(x^*)=(N-k+1)^{N-k+1} \eps^{N-k+1}/\lambda^{N-k+1}\cdot e^{-(N-k+1)}$. Therefore, 
    \begin{align*}
        &\left|\left(\sum_{k=0}^N \eps^k (T_k+\bar{T}_k)\right)^4 - \left(\sum_{k=0}^N \eps^k T_k\right)^4  - \left(\sum_{k=0}^N \eps^k (P_k+\bar{T}_k)\right)^4 + \left(\sum_{k=0}^N \eps^k P_k\right)^4 \right| \nonumber\\
        &\quad \le C \eps^{N+1} \sum_{k=0}^N \frac{(N-k+1)^{N-k+1}}{\lambda^{N-k+1}(N-k+1)!} e^{-(N-k+1)} = C \eps^{N+1} \sum_{n=1}^{N+1} \frac{n^n}{\lambda^n n!} e^{-n} \le C(\gamma_N -1)\eps^{N+1},
    \end{align*}
    where $\gamma_N:=\sum_{n=0}^{N+1}n^n/(\lambda^n n!)>1$ is a constant depending on $N$.
    Therefore, 
    \begin{align*}
        \|R_{131}\|_{L^2(\Omega)} \le C(\gamma_N-1) \eps^{N+1},~ \|R_{131}\|_{L^\infty(\Omega)} \le C(\gamma_N-1) \eps^{N+1}.
    \end{align*}
    Combining this with \eqref{eq:R132}, we get 
    \begin{align}\label{eq:R13est}
        \|R_{13}\|_{L^2(\Omega)} \le C\gamma_N \eps^{N+1},~\|R_{13}\|_{L^\infty(\Omega)} \le C\gamma_N \eps^{N+1}.
    \end{align}
    Finally, we consider $R_{14}$. Recalling \eqref{eq:ek0}, we have
    \begin{align}\label{eq:e00r}
        E_0^0& = (\tilde{T}_0-\tilde{T}_{0,\infty})\partial_\eta^2\chi + 2\partial_\eta\chi \partial_\eta \tilde{T}_0  + \chi \langle \tilde{T}_0^4-\tilde{T}_{0,\infty}^4\rangle - \langle (\chi(\tilde{T}_0-\tilde{T}_{0,\infty}) + \tilde{T}_{0,\infty})^4 - \tilde{T}_{0,\infty}^4 \rangle,
    \end{align}
    Since $\tilde{T}_0(\eta,x')-\tilde{T}_{0,\infty}(x')$ and $\partial_\eta \tilde{T}_0(\eta,x')$ exponentially decay to $0$ as $\eta\to\infty$ and $\partial_\eta \chi(\eps\eta) = \eps \partial_{x_1} \chi(x_1)$, $\partial_\eta^2\chi (\eps\eta)= \eps^2 \partial^2_{x_1}\chi(x_1)$ are supported on the interval $\eps\eta\in (\tfrac38\delta,\infty)$, hence 
    \begin{align}\label{eq:E00pre1}
        |(\tilde{T}_0-\tilde{T}_{0,\infty})\partial_\eta^2\chi + 2\partial_\eta\chi \partial_\eta \tilde{T}_0| \le C e^{-\lambda \eta}(\eps|\partial_{x_1}\chi| + \eps^2 |\partial_{x_1}\chi| )1_{\eta\ge \frac{3\delta}{8\eps}} \le C \eps e^{-\frac{3\lambda \delta}{8\eps} }.
    \end{align} 
    Since $\tilde{T}_0(\eta,x')-\tilde{T}_{0,\infty}(x')$ decays exponentially, 
    \begin{align*}
        |\tilde{T}_0^4 - \tilde{T}_{0,\infty}^4| = |(\tilde{T}_0-\tilde{T}_{0,\infty}| \cdot |(\tilde{T}_0+\tilde{T}_{0,\infty})(\tilde{T}_0^2+\tilde{T}_{0,\infty})^2| \le C e^{-\lambda \eta},
    \end{align*}
    and 
    \begin{align*}
        |(\chi(\tilde{T}_0-\tilde{T}_{0,\infty}) + \tilde{T}_{0,\infty})^4 - \tilde{T}_{0,\infty}^4| &= \chi|(\tilde{T}_0-\tilde{T}_{0,\infty})| \cdot |\chi(\tilde{T}_0-\tilde{T}_{0,\infty})+\tilde{T}_{0,\infty}|\cdot |(\chi(\tilde{T}_0-\tilde{T}_{0,\infty}))^2+\tilde{T}_{0,\infty}^2| \nonumber\\
        &\le C e^{-\lambda\eta}.
    \end{align*}
    Note that for $\eps \eta <\tfrac14$, $\chi(\eps\eta)=1$ and 
    \begin{align*}
        \langle \chi \langle \tilde{T}_0^4-\tilde{T}_{0,\infty}^4\rangle - \langle (\chi(\tilde{T}_0-\tilde{T}_{0,\infty}) + \tilde{T}_{0,\infty})^4 - \tilde{T}_{0,\infty}^4 \rangle =  \langle \tilde{T}_0^4-\tilde{T}_{0,\infty}^4\rangle - \langle \tilde{T}_0^4 - \tilde{T}_{0,\infty}^4 \rangle = 0.
    \end{align*}
    Hence the left term of the above equation is supported on the interval $\eps\eta\in [\tfrac14\delta,\infty)$. Therefore,
    \begin{align}\label{eq:E00pre2}
       | \langle \chi \langle \tilde{T}_0^4-\tilde{T}_{0,\infty}^4\rangle - \langle (\chi(\tilde{T}_0-\tilde{T}_{0,\infty}) + \tilde{T}_{0,\infty})^4 - \tilde{T}_{0,\infty}^4 \rangle| \le C e^{-\lambda \eta} 1_{\eta \ge \tfrac{\delta}{4\eps}} \le C e^{-\frac{\lambda \delta}{4\eps}}.
    \end{align}
    combining the above inequality with \eqref{eq:E00pre1}, \eqref{eq:e00r} satisfies 
    \begin{align}\label{eq:E00est}
        |E_0^0| \le C \eps e^{-\frac{3\delta}{8\eps}} + C e^{-\frac{\lambda \delta}{4\eps}} \le C e^{-\frac{\lambda \delta}{4\eps}}.
    \end{align}
    For $k\ge 1$, we recall \eqref{eq:ek0}:
    \begin{align*}
        & E_{k}^0= ((\tilde{T}_{k}-\tilde{T}_{k,\infty})\partial_\eta^2 \chi + 2 \partial_\eta\chi \partial_\eta \tilde{T}_k) + \Big(\chi(\langle 4\tilde{T}_0^3\tilde{T}_k - 4\tilde{T}_{0,\infty}^3\tilde{T}_{k,\infty} +   4(\tilde{T}_0^3-\tilde{T}_{0,\infty}^3)(P_k-P_k(0)) \rangle\nonumber\\
        &\quad+\langle\mathcal{E}(\bar{T}+P,k-1) - \mathcal{E}(P,k-1) \rangle)  -  \langle \mathcal{C}(\bar{T}+P,k) - \mathcal{C}(P,k)\rangle\Big) + (1-\chi\chi_0)\Delta_{x'}\bar{T}_{k-2} \nonumber\\
        &=: E_{k1}^0 + E_{k2}^0 + E_{k3}^0.  
    \end{align*}
    The term $ (\tilde{T}_{k}-\tilde{T}_{k,\infty})\partial_\eta^2 \chi + 2 \partial_\eta\chi \partial_\eta \tilde{T}_k $ can be estimated in the same  way as \eqref{eq:E00pre1} as 
    \begin{align}\label{eq:Ek10pre}
        |E_{k1}^0|=| (\tilde{T}_{k}-\tilde{T}_{k,\infty})\partial_\eta^2 \chi + 2 \partial_\eta\chi \partial_\eta \tilde{T}_k | \le C\eps  e^{-\frac{3\lambda \delta}{8\eps}}.
    \end{align}
    For $E_{k3}^0$, $(1-\chi(\eps\eta)\chi_0(\eps\eta))$ is supported in $\eps\eta\in [\tfrac14\delta,\infty)$. By \eqref{eq:bdecay},
    \begin{align}\label{eq:Ek30pre}
        |E_{k3}^0| \le Ce^{-\lambda\eta} 1_{\eta>\tfrac{\delta}{4\eps}} \le C e^{-\frac{\lambda \delta}{4\eps}}.
    \end{align}
    To estimate $E_{k2}^2$, due to \eqref{eq:bdecay}, 
    \begin{align*}
        |4\tilde{T}_0^3\tilde{T}_k - 4\tilde{T}_{0,\infty}^3\tilde{T}_{k,\infty} +   4(\tilde{T}_0^3-\tilde{T}_{0,\infty}^3)(P_k-P_k(0)) | \le C e^{-\lambda\eta} + Ce^{-\lambda\eta} (1+\eta^k) \le C(1+\eta^k)e^{-\lambda\eta}.
    \end{align*}
    Since $(\mathcal{E}(\bar{T}+P,k-1) - \mathcal{E}(P,k-1))/\bar{T}$ is a polynomial of $\bar{T}$, $P$ of order no bigger than $\eta^{k-1}$ and thus is bounded, we have
    \begin{align*}
        |\mathcal{E}(\bar{T}+P,k-1) - \mathcal{E}(P,k-1)| = |\bar{T}| \left|\frac{ \mathcal{E}(\bar{T}+P,k-1) - \mathcal{E}(P,k-1)}{\bar{T}}\right| \le C(1+\eta^{k-1}) e^{-\lambda\eta}.
    \end{align*}
   Similarly,
    \begin{align*}
        |\mathcal{C}(\bar{T}+P,k) - \mathcal{C}(P,k)| = |\bar{T}| \left|\frac{ \mathcal{C}(\bar{T}+P,k) - \mathcal{C}(C,k)}{\bar{T}}\right| \le C(1+\eta^{k}) e^{-\lambda\eta}.
    \end{align*}
    Therefore,
    \begin{align*}
        E_{k2}^0 \le C (1+\eta^{k})e^{-\lambda\eta}.
    \end{align*}
    Moreover, when $\eps\eta<\tfrac14$, $\chi(\eps\eta)=1$ and due to the formula $\mathcal{C}(\bar{T}+P,k)-\mathcal{C}(P,k) = 4(\bar{T}_0+P_0)^3(\bar{T}_k+P_k) - 4P_0^3P_k + \mathcal{E}(\bar{T}+P,k-1)-\mathcal{E}(P,k-1)$
   as well as 
   \begin{align*}
       &4\tilde{T}_0^3\tilde{T}_k - 4\tilde{T}_{0,\infty}^3\tilde{T}_{k,\infty} + 4(\tilde{T}_0^3-\tilde{T}_{0,\infty}^3)(P_k-P_k(0)) \nonumber\\
       &\quad = 4(\bar{T}_0+P_0)^3(\bar{T}_k + P_k(0)) - 4\tilde{T}_{0,\infty}^3 P_k(0) + 4((\bar{T}_0+P_0)^3-\tilde{T}_{0,\infty}^3)(P_k-P_k(0)) \nonumber\\
       &\quad = 4 (\bar{T}_0+P_0)^3(\bar{T}_k+P_k) - 4P_0^3P_k,
   \end{align*} 
   where we use the fact that $P_0=T_0(0)=\tilde{T}_{0,\infty}$, we have $E_{k2}^0=0$ on the interval $\eps\eta\in [0,\tfrac14\delta)$. Thus 
   \begin{align*}
       E_{k2}^0 \le C (1+\eta^{k}) e^{-\lambda\eta}1_{\eta>\tfrac{\delta}{4\eps}} \le C(1+\frac{1}{\eps^k}) e^{-\frac{\lambda \delta}{4\eps}}.
   \end{align*}
   combining this with \eqref{eq:Ek10pre} and \eqref{eq:Ek30pre}, we obtain 
   \begin{align*}
      |E_{k}^0| \le C\eps e^{-\frac{3\delta}{8\eps}} + Ce^{-\frac{\lambda \delta}{4\eps}} +  C(1+\frac{1}{\eps^k}) e^{-\frac{\lambda \delta}{4\eps}} \le C e^{-\frac{\lambda \delta}{4\eps}} + C \frac{1}{\eps^k} e^{-\frac{\lambda \delta}{4\eps}}.
   \end{align*}
   combining the above estimate with \eqref{eq:E00est}, we arrive at 
   \begin{align*}
       \left|\sum_{k=0}^N E_{k}^0\right| \le \sum_{k=0}^N |E_k^0| \le Ce^{-\frac{\lambda \delta}{4\eps}} + \sum_{k=1}^N \eps^k\left(1+\frac{1}{\eps^k}\right) e^{-\frac{\lambda \delta}{4\eps}} \le (N+1)C e^{-\frac{\lambda\delta}{4\eps}} + e^{-\frac{\lambda \delta}{4\eps}}\sum_{k=1}^N \eps^k.
   \end{align*}
   We can take $\eps<1/2$ so that $\sum_{k=1}^N \eps^k\le \sum_{k=1}^N 1/2^k \le 1$ and get 
   \begin{align*}
    |R_{14}|=\left|\sum_{k=0}^N E_{k}^0\right| \le (N+2) Ce^{-\frac{\lambda \delta}{4\eps}}.
   \end{align*}
   By the regularity of solutions, we have 
   \begin{align*}
       \|R_{14}\|_{L^2(\Omega)},~\|R_{14}\|_{L^\infty(\Omega)} \le (N+2) Ce^{-\frac{\lambda \delta}{4\eps}}.
   \end{align*}
   Taking the above inequality and \eqref{eq:R11est}, \eqref{eq:R12est} and \eqref{eq:R13est}, we obtain 
   \begin{align}\label{eq:R1/est}
       \|\mathcal{R}_1(T^a,\psi^a)\|_{L^2(\Omega)}, ~\|\mathcal{R}_1(T^a,\psi^a)\|_{L^\infty(\Omega)} \le C\gamma_N \eps^{N+1} + (N+2) Ce^{-\frac{\lambda \delta}{4\eps}}.
   \end{align}
   
   One can estimate $\mathcal{R}_2(T^a,\psi^a)$ given by \eqref{eq:R2b} in the same way. Recalling \eqref{eq:R2b},
   \begin{align*}
    \mathcal{R}_2(T^a,\psi^a) &=  (\eps^{N+1} \beta\cdot \nabla \psi_{N} +\varepsilon^{N+1} \beta'\cdot\nabla_{x'} \bar{\psi}_N) - \sum_{k=N+1}^{4N} \eps^k(\mathcal{C}(T+\bar{T},k))  \nonumber\\
    &\quad - \sum_{k=0}^N \eps^k ( \mathcal{C}(T+\bar{T},k) - \mathcal{C}(T,k) - \mathcal{C}(\bar{T}+P,k) + \mathcal{C}(P,k)) + \sum_{k=0}^N \eps^k E_k^1 \nonumber\\
    &=:R_{21}+R_{22}+R_{23}+R_{24}.
   \end{align*}
 By the boundness of solutions to the interior expansion and boundary layer corrections, $|R_{21}| = O(\eps^{N+1})$, $|R_{22}|=O(\eps^{N+1})$. The term $R_{23}$ is the same as that of $R_{13}$ without integration over $\beta\in\mathbb{S}^2$ and thus can be estimated in the same way, $|R_{13}|\le C\gamma_N\eps^{N+1}$. Finally, $R_{24}$ can be estimated in the same way as $R_{14}$, whereas the only difference is the term $\mu (\tilde{\psi}_k-\tilde{\psi}_{k,\infty})\partial_\eta \chi$, which is supported on $\eps\eta \in [\tfrac14\delta,\infty)$ and thus can be estimate by 
   \begin{align*}
       |\mu (\tilde{\psi}_k-\tilde{\psi}_{k,\infty})\partial_\eta \chi| \le Ce^{-\lambda\eta} 1_{\eta\ge \tfrac{\delta}{4\eps}} \le Ce^{-\lambda \delta/4\eps},
   \end{align*}
   and $|R_{24}| \le (N+2)Ce^{-\frac{\lambda}{4\eps}}$. Thus \eqref{eq:R1/est} also holds for $\mathcal{R}_{2}(T^a,\psi^a)$, i.e.
   \begin{align*}
    \|\mathcal{R}_2(T^a,\psi^a)\|_{L^2(\Omega\times\mathbb{S}^2)}, ~\|\mathcal{R}_2(T^a,\psi^a)\|_{L^\infty(\Omega\times\mathbb{S}^2)} \le C\gamma_N \eps^{N+1} + (N+2) Ce^{-\frac{\lambda \delta}{4\eps}}.
   \end{align*}
\end{proof}
\subsection{An inequality}
Next we prove inequality \eqref{eq:est/gg} under the spectral assumption.
\begin{lemma} \label{lm.g2}
    Let $(T^a,\psi^a)$ be the approximate solution constructed in the previous section. Assuming the spectral assumption \ref{asA} holds for the solution $\tilde{T}_0$ of the nonlinear Milne problem \eqref{eq:ml1}-\eqref{eq:ml2b} where $\tilde{T}_0\ge a$ for some constant $a>0$.
    Then, for $\eps>0$ sufficiently small, the following inequality holds
    \begin{align}\label{eq:est>g2}
       - \int_\Omega 4(T^a)^3 g \Delta g dx = \int_\Omega 4(T^a)^3|\nabla g|^2 dx - \int_\Omega \nabla (4(T^a)^3)\cdot g \nabla g dx  \ge \kappa \int_\Omega |\nabla g|^2 dx  -C \|g\|_{L^2(\Omega)}^2,
    \end{align}
for any function $g$ satisfying $g(0)=0$ and for some constants $\kappa>0$, $C>0$ depending on $M$, where $M<1$ is the constant in \eqref{eq:sp} of the spectral assumption.
\end{lemma}
\begin{proof}
    Note that $T^a=\sum_{k=0}^N \eps^k(T_k + \bar{T}_k)$ where $\bar{T}_k = \chi(x_1)(\tilde{T}_k-\tilde{T}_{k,\infty})$. In the domain $x_1>\tfrac38 \delta$, $\chi(x_1)=0$ and $T^a = \sum_{k=0}^N \eps^k T_k$, which only contain the interior approximations. Since $\|T_k\|_{C^s(\Omega)}$ is bounded for any $s>0$ and $k=1,\ldots,N$,
    \begin{align}
        &\int_{\Omega\cap \{x_1>\tfrac38\delta\}} 4(T^a)^3|\nabla g|^2 dx - \int_{\Omega\cap \{x_1>\tfrac38\delta\}}  \nabla (4(T^a)^3)\cdot g \nabla g dx \nonumber\\
        &\quad = \int_{\Omega\cap \{x_1>\tfrac38\delta\}} 4(T^a)^3|\nabla g|^2 dx - \int_{\Omega\cap \{x_1>\tfrac38\delta\}} 2 (T^a)^{3/2}\nabla g \cdot 6 (T^a)^{1/2} \nabla T^a  g dx \nonumber\\
        &\quad \ge \int_{\Omega\cap \{x_1>\tfrac38\delta\}} 4(T^a)^3|\nabla g|^2 dx - \frac12 \int_{\Omega\cap \{x_1>\tfrac38\delta\}} 4(T^a)^3|\nabla g|^2 dx - \frac12 \int_{\Omega\cap \{x_1>\tfrac38\delta\}} 36 T^a |\nabla T^a|^2 g^2 dx \nonumber\\
        &\quad \ge \int_{\Omega\cap \{x_1>\tfrac38\delta\}} 2(T^a)^3|\nabla g|^2 dx - C \|g\|_{L^2(\Omega)}^2. \label{eq:366}
    \end{align}
    In the domain $x_1\le \frac{3}{8}\delta$, boundary layer effects play a role. First we split the integral as
    \begin{align*}
        &\int_{\Omega\cap \{x_1\le\tfrac38\}} 4(T^a)^3|\nabla g|^2 dx - \int_{\Omega\cap \{x_1\le\tfrac38\}}  \nabla (4(T^a)^3)\cdot g \nabla g dx \nonumber\\
        &\quad = \int_{\Omega\cap \{x_1\le\tfrac38\}} 4(T^a)^3|\nabla_{x'} g|^2 dx - \int_{\Omega\cap \{x_1\le\tfrac38\}}  \nabla_{x'} (4(T^a)^3)\cdot g \nabla_{x'} g dx \nonumber\\
        &\qquad + \int_{\Omega\cap \{x_1\le\tfrac38\}} 4(T^a)^3|\partial_{x_1} g|^2 dx - \int_{\Omega\cap \{x_1\le\tfrac38\}}  \partial_{x_1} (4(T^a)^3)\cdot g \partial_{x_1} g dx
        =: I_1 + I_2.
    \end{align*}
    Since $\|\nabla_{x'}(T^a)\|_{L^2(\Omega)}$ is bounded, we can estimate $I_1$ the same as \eqref{eq:366}:
    \begin{align}\label{eq:I1/e}
        I_1 \ge \int_{\Omega\cap \{x_1\le \tfrac38\}} 2(T^a)^3|\nabla_{x'} g|^2 dx - C \|g\|_{L^2(\Omega)}^2.
    \end{align} 
    To estimate $I_2$, we use the spectral assumption \ref{asA}. Near the boundary, the composite approximate solution $(T^a,\psi^a)$ is close to the solution $(\tilde{T}_0,\tilde{\psi}_0)$ of the nonlinear Milne problem \ref{eq:ml1}-\eqref{eq:ml2b}.
    Using the equation
    \begin{align*}
        (T^a)^3= \left(\sum_{k=0}^N \eps^k (T_k+\bar{T}_k) \right)^3 = (T_0+\bar{T}_0)^3 + \varepsilon G,
    \end{align*}
    where $G=3 (T_0+\bar{T}_0)^2 \sum_{k=1}^N \eps^{k-1} (T_k+\bar{T}_k) + 6(T_0+\bar{T}_0)^2 (\sum_{k=1}^N \eps^{k-1} (T_k+\bar{T}_k))^2 + 3(T_0+\bar{T}_0) (\sum_{k=1}^N \eps^{k-1} (T_k+\bar{T}_k))^2$,
    we can rewrite $I_2$ as 
    \begin{align*}
        I_2 &=\frac{1}{\varepsilon^2} \int_{\mathbb{T}^2} \int_0^{\frac{3}{8\eps}}  4(T_0+\bar{T}_0)^3 + 4 \varepsilon G) |\partial_\eta g|^2 d\eta dx' -\frac{1}{\varepsilon^2}\int_{\mathbb{T}^2} \int_0^{\frac{3}{8\eps}} \partial_{\eta} (4(T_0+\bar{T}_0)^3 + 4 \varepsilon G) g \partial_\eta g d\eta dx' \nonumber\\
        &= \frac{1}{\varepsilon^2}\int_{\mathbb{T}^2} dx' \int_{0}^{\frac{3}{8\varepsilon}} (4\tilde{T}_0^3 |\partial_\eta g|^2 - \partial_\eta (4 \tilde{T}_0^3)g\partial_\eta g )d\eta + \frac{1}{\varepsilon^2} \int_{\mathbb{T}^2} dx' \int_{0}^{\frac{3}{8\varepsilon}} (4(T_0+\bar{T}_0)^3-4\tilde{T}_0^3)|\partial_\eta g|^2 d\eta \nonumber\\
        &\quad - \frac{1}{\varepsilon^2} \int_{\mathbb{T}^2}dx' \int_0^{\frac{3}{8\varepsilon}} \partial_\eta (4(T_0+\bar{T}_0)^3 - 4\tilde{T}_0^3) g\partial_\eta g d\eta dx' + \frac{1}{\varepsilon^2} \int_{\mathbb{T}^2}dx' \int_0^{\frac{3}{8\varepsilon}}  \varepsilon 4 (G |\partial_\eta g|^2 - \partial_\eta G g \partial_\eta g) d\eta \nonumber\\
        &=:I_{21}+I_{22}+I_{23}+I_{24}.
    \end{align*}
    The spectral assumption \ref{asA} implies  
    \begin{align*}
        I_{21} &= \frac{1}{\varepsilon^2}\int_{\mathbb{T}^2} dx' \int_{0}^{\frac{3}{8\varepsilon}} (4\tilde{T}_0^3 |\partial_\eta g|^2 - \partial_\eta (4 \tilde{T}_0^3)g\partial_\eta g )d\eta \nonumber\\
        &\ge \frac{1}{\varepsilon^2}\int_{\mathbb{T}^2} dx' \int_{0}^{\frac{3}{8\varepsilon}}( 4\tilde{T}_0^3 |\partial_\eta g|^2 - \frac12 (4\tilde{T}_0^3 |\partial_\eta g|^2+36\tilde{T}_0|\partial_\eta \tilde{T}_0|^2 g^2) ) d\eta \nonumber\\
        &\ge  \frac{1}{2\varepsilon^2}\int_{\mathbb{T}^2} dx' \int_{0}^{\frac{3}{8\varepsilon}} (4 \tilde{T}_0^3 |\partial_\eta g|^2 - 36\tilde{T}_0|\partial_\eta \tilde{T}_0|^2 g^2 ) d\eta \nonumber\\
        &\ge \frac{1-M}{2\varepsilon^2} \int_{\mathbb{T}^2} dx' \int_{0}^{\frac{3}{8\varepsilon}} 4\tilde{T}_0^3|\partial_\eta g|^2 d\eta.
    \end{align*}
    For $I_{22}$, since $\bar{T}_0=T_0 + \chi(\varepsilon\eta) (\tilde{T}_0-T_0(0))$, it holds that 
    \begin{align*}
        (T_0+\bar{T}_0)^3 - \tilde{T}_0^3 &= (T_0 +\chi(\varepsilon\eta) (\tilde{T}_0-T_0(0)) -\tilde{T}_0)(\tilde{T}_0^2+\tilde{T}_0(T_0+\bar{T}_0)+(T_0+\bar{T}_0)^2) \nonumber\\
        &=((T_0-T_0(0)) - (1-\chi(\varepsilon\eta))(\tilde{T}_0 - \bar{T}_0))(\tilde{T}_0^2+\tilde{T}_0(T_0+\bar{T}_0)+(T_0+\bar{T}_0)^2) \nonumber\\
        &= (\partial_{x_1}T_0(\xi) \varepsilon \eta  - (1-\chi(\varepsilon\eta))(\tilde{T}_0 - \bar{T}_0))(\tilde{T}_0^2+\tilde{T}_0(T_0+\bar{T}_0)+(T_0+\bar{T}_0)^2).
    \end{align*}
    Since we are considering the integration over $x_1=\varepsilon\eta \in [0,\tfrac38\delta]$ and  $(1-\chi(\varepsilon\eta))$ is supported on $[\tfrac14\delta,\tfrac38\delta]$,
    \begin{align*}
        I_{22} = \frac{1}{\varepsilon^2} \int_{\mathbb{T}^2} dx' \int_{0}^{\frac{3}{8\varepsilon}} (4(T_0+\bar{T}_0)^3-4\tilde{T}_0^3)|\partial_\eta g|^2 d\eta \le \frac{3\delta C}{8\varepsilon^2} \int_{\mathbb{T}^2} dx' \int_0^{\frac{\delta}{2\eps}} |\partial_\eta g|^2 d\eta.
    \end{align*}
    For $I_{23}$, due to 
    \begin{align*}
        \partial_\eta  ((T_0+\bar{T}_0)^3 - \tilde{T}_0^3) &= \partial_\eta(T_0 +\chi(\varepsilon\eta) (\tilde{T}_0-T_0(0)) -\tilde{T}_0)(\tilde{T}_0^2+\tilde{T}_0(T_0+\bar{T}_0)+(T_0+\bar{T}_0)^2) \nonumber\\
        &\quad  + (T_0 +\chi(\varepsilon\eta) (\tilde{T}_0-T_0(0)) -\tilde{T}_0)\partial_\eta (\tilde{T}_0^2+\tilde{T}_0(T_0+\bar{T}_0)+(T_0+\bar{T}_0)^2) \nonumber\\
        &= \varepsilon \chi'(\eps\eta)(\tilde{T}_0-{T}_0(0))(\tilde{T}_0-T_0(0)) -\tilde{T}_0)(\tilde{T}_0^2+\tilde{T}_0(T_0+\bar{T}_0)+(T_0+\bar{T}_0)^2) \nonumber\\
        &\quad + (\chi(\varepsilon\eta)-1) \partial_\eta \tilde{T}_0 (\tilde{T}_0^2+\tilde{T}_0(T_0+\bar{T}_0)+(T_0+\bar{T}_0)^2) \nonumber\\
        &\quad +  (\partial_{x_1}T_0(\xi) \varepsilon \eta  - (1-\chi(\varepsilon\eta))(\tilde{T}_0 - \bar{T}_0))\partial_\eta(\tilde{T}_0^2+\tilde{T}_0(T_0+\bar{T}_0)+(T_0+\bar{T}_0)^2),
    \end{align*}
    with consideration of $\eps\eta\in [0,\tfrac38\delta]$ and and  $(1-\chi(\varepsilon\eta))$ being supported on $[\tfrac14\delta,\tfrac38\delta]$,
    it holds that 
    \begin{align*}
        I_{23} &= - \frac{1}{\varepsilon^2} \int_{\mathbb{T}^2}dx' \int_0^{\frac{3\delta}{8\varepsilon}} \partial_\eta (4(T_0+\bar{T}_0)^3 - 4\tilde{T}_0^3) g\partial_\eta g d\eta  \le  \frac{C}{\varepsilon^2} \int_{\mathbb{T}^2}dx' \int_0^{\frac{3\delta}{8\varepsilon}} (\varepsilon |g| |\partial_\eta g| + \frac{3\delta}{8} |g| |\partial_\eta g| )d\eta \nonumber\\
        &\le \frac{3\delta C}{8\eps^2} \int_{\mathbb{T}^2}dx' \int_0^{\frac{3\delta}{8\varepsilon}} (g^2 + |\partial_\eta g|^2) d\eta.
    \end{align*}
    For $I_{24}$, we have 
    \begin{align*}
        I_{24} =  \frac{1}{\varepsilon^2} \int_{\mathbb{T}^2}dx' \int_0^{\frac{3\delta}{8\varepsilon}}  \varepsilon 4 (G |\partial_\eta g|^2 - \partial_\eta G g \partial_\eta g) d\eta \le \frac{C }{\varepsilon} \int_{\mathbb{T}^2}dx' \int_0^{\frac{3\delta}{8\varepsilon}}( |g|^2 + |\partial_\eta g|^2) d\eta.
    \end{align*}
    combining the above estimates gives 
    \begin{align*}
        I_2&\ge \frac{1-M}{2\varepsilon^2} \int_{\mathbb{T}^2} dx' \int_{0}^{\frac{3\delta}{8\varepsilon}} 4\tilde{T}_0^3|\partial_\eta g|^2 d\eta - \frac{3\delta C}{8\varepsilon^2}\int_{\mathbb{T}^2}dx' \int_0^{\frac{3\delta}{8\varepsilon}}( |g|^2 + |\partial_\eta g|^2) d\eta \nonumber\\
        &\quad -  \frac{C}{\varepsilon} \int_{\mathbb{T}^2}dx' \int_0^{\frac{3\delta}{8\varepsilon}}( |g|^2 + |\partial_\eta g|^2) d\eta.
    \end{align*}
    By the assumption of the lemma, $\tilde{T}_0\ge a$, hence $4\tilde{T}_0^3 \ge 4a^3$ for some constant $a>0$.
    We can take sufficiently small $\eps$ and $\delta$ such that $\varepsilon<(1-M)a^3/C$ and $3\delta C/8 \le (1-M)/8$, and we get from the above inequality
    \begin{align*}
        I_2 \ge \frac{1-M}{4\varepsilon^2} \int_{\mathbb{T}^2} dx' \int_{0}^{\frac{3\delta}{8\varepsilon}} 4\tilde{T}_0^3|\partial_\eta g|^2 d\eta. 
    \end{align*}
    combining this with \eqref{eq:366} and \eqref{eq:I1/e} implies
    \begin{align*}
        - \int_\Omega 4(T^a)^3 g \Delta g dx &= \int_\Omega 4(T^a)^3|\nabla g|^2 dx - \int_\Omega \nabla (4(T^a)^3)\cdot g \nabla g dx \nonumber\\
        &\ge\int_{\Omega\cap \{x_1>\tfrac38\delta\}} 2(T^a)^3|\nabla g|^2 dx  + \int_{\Omega\cap \{x_1\le \tfrac38\delta\}} 2(T^a)^3|\nabla_{x'} g|^2 dx  \nonumber\\
        &\quad +\frac{1-M}{4\varepsilon^2} \int_{\mathbb{T}^2} dx' \int_{0}^{\frac{3\delta}{8\varepsilon}} 4\tilde{T}_0^3|\partial_\eta g|^2 d\eta - C\|g\|_{L^2(\Omega)}^2\nonumber\\
        &\ge \kappa \|\nabla g\|_{L^2(\Omega)}^2-C\|g\|_{L^2(\Omega)}^2,
    \end{align*}
    where $\kappa=\min\{2a^3,(1-M)a^3\}$, which finishes the proof of Lemma \ref{lm.g2}.
\end{proof}

\section{Diffusive Limit}\label{sec4}

In this section, we prove Theorem \ref{thm.1} by estimating the difference between the solution $(T^\eps,\psi^\eps)$ to system \eqref{eq:1}-\eqref{eq:2} and the constructed approximate solution $(T^a,\psi^a)$, which satisfies \eqref{eq:1a}-\eqref{eq:2a}.
 Setting $g:=T^\eps-T^a,\phi := \psi^\eps - \psi^a$, functions $(g,\phi)$ then satisfy
\begin{align}
    &\eps^2 \Delta g + \langle \phi - (T^a + g)^4 + (T^a)^4 \rangle = - \mathcal{R}_1(T^a,\psi^a),\label{eq:n1}\\
    &\eps \beta\cdot\nabla\phi + \phi -  (T^a + g)^4 + (T^a)^4 = -\mathcal{R}_2(T^a,\psi^a),\label{eq:n2}
\end{align}
with boundary conditions
\begin{align*}
    &g(x) = 0, \quad \text{for } x \in \partial\Omega,\\
    &\phi(x,\beta)=0,\text{ for }(x,\beta)\in \Gamma_-.
\end{align*}
In order to prove Theorem \ref{thm.1}, we first derive suitable estimates on a linearized system and then use Banach fixed point theorem to show the existence of the above problem near zero solutions, leading to the convergence of $\left(T^\eps,\psi^\eps\right)$ to $\left(T^a,\psi^a\right)$ as $\eps\to 0$.

\subsection{Linearized system} 
 We first consider the following linear system:
\begin{align}
    &\eps^2 \Delta g + \langle \phi - 4(T^a)^3g  \rangle = r_1 + \langle r \rangle, \label{eq:l1}\\
    &\eps \beta\cdot \nabla \phi+\phi - 4(T^a)^3 g = r_2+ r,\label{eq:l2}
\end{align}
where $r_1 = r_1(x), r=r(x,\beta)$, and $r_{2}=r_2(x,\beta)$ are given functions and the boundary conditions are taken to be 
\begin{align}
    &g(x) = 0,\text{ for any }x\in\partial\Omega, \label{eq:l1/b}\\
    &\phi(x,\beta) = 0,\text{ for any }(x,\beta)\in \Gamma_{-}.\label{eq:l2/b}
\end{align}
The existence of solutions to the above problem and suitable estimates on the solutions are stated in the following lemma.
\begin{lemma}\label{lm:linear}
    Let $\eps>0$, $(T^a,\psi^a)$ be the composite approximate solution constructed in section \ref{sec2}. Assuming  $r_1 \in L^2\cap L^\infty(\Omega)$, and $r, r_2 \in L^2\cap L^\infty(\Omega\times\mathbb{S}^2)$. Then, there exists a unique solution $(g,\phi) \in L^2 \cap L^\infty(\Omega) \times L^2 \cap L^\infty(\Omega\times\mathbb{S}^2)$ to system \eqref{eq:l1}-\eqref{eq:l2} with boundary conditions \eqref{eq:l1/b}-\eqref{eq:l2/b}. Moreover, the solution $(g,\phi)$ satisfies the following estimates
    \begin{align}\label{eq:l2result}
        &\eps \|\phi\|_{L^2(\Omega\times\mathbb{S})} + \eps \| g\|_{H^1(\Omega)}+\sqrt{\eps} \|\phi\|_{L^2(\Gamma_+)} + \|\phi-4(T^a)^3 g\|_{L^2(\Omega\times \mathbb{S}^2)} \nonumber\\
        &\quad \le C\|r\|_{L^2(\Omega\times\mathbb{S}^2)} + \frac{C}{\eps}(\|r_1\|_{L^2(\Omega)} + \|r_2\|_{L^2(\Omega\times\mathbb{S}^2)}).    
    \end{align}
    and
    \begin{align}\label{eq:linfresult}
        &\|\phi\|_{L^\infty(\Omega\times\mathbb{S}^2)}+\|g\|_{L^\infty(\Omega)} \nonumber\\
        &\quad \le \frac{C}{\eps^{2}} \|r\|_{L^2(\Omega\times\mathbb{S}^2)}+ \frac{C}{\eps^3}(\|r_1\|_{L^2(\Omega)}+ \|r_2\|_{L^2(\Omega\times\mathbb{S}^2)}  ) + C\|r_2\|_{L^\infty(\Omega\times\mathbb{S}^2)}+C\|r\|_{L^\infty(\Omega)},    
    \end{align}
    where $C>0$ is a constant not depending on $\eps$.
\end{lemma}
\begin{proof}
    Existence of the linear system \eqref{eq:l1}-\eqref{eq:l2} with homogeneous boundary conditions \eqref{eq:l1/b}-\eqref{eq:l2/b} follows from standard theory of elliptic and transport equations. To derive the estimates \eqref{eq:l2result} and \eqref{eq:linfresult}, we  first derive the energy estimate. Then the $L^2$ type estimate is derived. Finally, the $L^\infty$ type estimate is shown.

\emph{{Step 1: The energy estimate.}} 
We multiply \eqref{eq:l1} by $4(T^a)^3 g$ and \eqref{eq:l2} by $\phi$, and integrate over $x\in \Omega$ and $\beta\in \mathbb{S}^2$ to get 
\begin{align} \label{eq:e1}
    &-\int_{\Omega } \eps^2 4(T^a)^3 g \Delta g dx + \iint_{\Omega\times\mathbb{S}^2} \eps \beta\cdot \nabla \frac{\phi^2}{2} d\beta dx + \iint_{\Omega\times\mathbb{S}^2} (\phi-4(T^a)^3 g)^2 d\beta dx \nonumber \\
    &\phantom{xx}{} = \iint_{\Omega\times\mathbb{S}^2} \phi r_2 d\beta dx - \int_\Omega 4(T^a)^3 gr_1 dx + \iint_{\Omega\times\mathbb{S}^2} (\phi - 4(T^a)^3 g) rd\beta dx.
\end{align} 
The boundary condition \eqref{eq:l2/b} implies 
\begin{align}\label{eq:e2}
    \iint_{\Omega\times\mathbb{S}^2} \eps \beta\cdot \nabla \frac{\phi^2}{2} d\beta dx &= \eps \iint_{\Gamma} \beta\cdot n \phi^2 d\beta d\sigma_x = \eps \iint_{\Gamma_+} \beta\cdot n \phi^2 d\beta d\sigma_x = \eps \|\phi\|_{L^2(\Gamma_+)}^2 .
\end{align}
By Lemma \ref{lm.g2}, inequality \eqref{eq:est>g2} gives
\begin{align}\label{eq:e3}
    -\int_{\Omega} 4(T^a)^3 g \Delta g dx \ge \kappa \|\nabla g\|_{L^2(\Omega)}^2 -C\|g\|_{L^2(\Omega)}^2. 
\end{align}
Applying Young's inequality on the last term of \eqref{eq:e1} gives
\begin{align*}
    \iint_{\Omega\times\mathbb{S}^2} (\phi - 4(T^a)^3 g) rd\beta dx \le \frac{1}{2} \iint_{\Omega\times\mathbb{S}^2} (\phi-4(T^a)^3 g)^2 d\beta dx + \frac{1}{2} \iint_{\Omega\times\mathbb{S}^2} r^2 d\beta dx.
\end{align*}
Taking the above inequality and \eqref{eq:e2},\eqref{eq:e3} into \eqref{eq:e1}, we obtain the following energy estimate: 
\begin{align}\label{eq:energyest}
    &\eps^2 \kappa \|\nabla g\|_{L^2(\Omega)}^2+\eps \|\phi\|_{L^2(\Gamma_+)}^2 + \frac12 \|\phi-4(T^a)^3 g\|_{L^2(\Omega\times \mathbb{S}^2)}^2 \nonumber \\
    &\quad \le C \eps^2 \|g\|_{L^2(\Omega)}^2 + \iint_{\Omega\times\mathbb{S}^2} \phi r_2 d\beta dx - \int_\Omega 4(T^a)^3 gr_1 dx + \frac12 \|r\|_{L^2(\Omega\times\mathbb{S}^2)}^2.
\end{align}
\bigskip

\emph{{Step 2:  The $L^2$ estimate.}} 
First we estimate the $L^2$ norm of $g$. Given $\rho=\rho(x,\beta) \in L^2(\Omega\times\mathbb{S}^2)$, define the operator $\mathcal{A}:L^2(\Omega)\mapsto L^2(\Omega)$ by 
 \begin{align}\label{eq:opA}
     \mathcal{A} h = \langle \varphi - h\rangle,\quad \text{ where }\varphi \text{ solves } 
     \left\{\begin{array}{cl}
        \varepsilon \beta \cdot \nabla \varphi + \varphi - h= \rho, &\text{ in }\Omega\times\mathbb{S}^2,\\
        \varphi(x,\beta)=0,&\text{ for } (x,\beta)\in \Gamma_-,
     \end{array}\right.
 \end{align}
 for $h=h(x) \in L^2(\Omega)$.
Then for any function $\ell=\ell(x) \in C^2(\Omega)$ satisfying $\ell=0$ on $\partial\Omega$, 
\begin{align*}
    \int_\Omega \mathcal{A} h \cdot \ell dx &= \int_\Omega \int_{\mathbb{S}^2} (\varphi - h) \ell d\beta dx \nonumber\\
    &= - \int_\Omega \int_{\mathbb{S}^2} \varepsilon \beta\cdot \nabla \varphi \ell d\beta dx + \int_\Omega\int_{\mathbb{S}^2} \rho \ell d\beta dx \nonumber\\
    &=\int_\Omega \int_{\mathbb{S}^2} \varepsilon \varphi \beta \cdot \nabla \ell d\beta dx - \varepsilon\iint_{\Gamma} \beta\cdot n \varphi \ell d\beta d\sigma_x + \int_\Omega\int_{\mathbb{S}^2} \rho \ell d\beta dx \nonumber\\
    &= \eps \int_\Omega\int_{\mathbb{S}^2} (\varphi - h) \beta \cdot \nabla \ell d\beta dx - 0 + \int_\Omega\int_{\mathbb{S}^2} \rho \ell d\beta dx \nonumber\\
    & =-\eps\int_\Omega \int_{\mathbb{S}^2} \varepsilon \beta \cdot \nabla \varphi \beta\cdot \nabla \ell d\beta dx +\eps \int_\Omega \int_{\mathbb{S}^2} \rho \beta\cdot \nabla \ell d\beta dx + \int_\Omega\int_{\mathbb{S}^2} \rho \ell d\beta dx  \nonumber\\
    & = \eps^2 \int_\Omega \int_{\mathbb{S}^2}  (\varphi - h) (\beta\cdot \nabla)^2 \ell d\beta dx + \eps^2 \int_\Omega \int_{\mathbb{S}^2} h (\beta \cdot \nabla)^2 \ell  d\beta dx \nonumber\\
    &\quad -\eps^2 \int_{\Gamma_+} \beta\cdot n \varphi \beta\cdot \nabla \ell d\beta d\sigma_x+ \eps\int_\Omega \int_{\mathbb{S}^2} \rho \beta\cdot \nabla \ell d\beta dx+ \int_\Omega\int_{\mathbb{S}^2} \rho \ell d\beta dx \nonumber\\
    & = -\eps^3 \int_{\Omega}\int_{\mathbb{S}^2} (\beta \cdot\nabla)^2 \ell \beta\cdot \nabla \varphi d\beta dx + \eps^2 \int_\Omega\int_{\mathbb{S}^2} \rho (\beta\cdot\nabla)^2 \ell d\beta dx
     + \eps^2 \frac{4\pi}{3} \int_\Omega h \Delta \ell dx \nonumber \\&\quad-\eps^2 \int_{\Gamma_+} \beta\cdot n \varphi \beta\cdot \nabla \ell d\beta d\sigma_x + \eps \int_\Omega \int_{\mathbb{S}^2} \rho \beta\cdot \nabla \ell d\beta dx+ \int_\Omega\int_{\mathbb{S}^2} \rho \ell d\beta dx \nonumber \\
    &=\eps^3 \int_{\Omega}\int_{\mathbb{S}^2} (\varphi - \langle \varphi\rangle/4\pi) (\beta\cdot\nabla)^3 \ell d\beta dx - \eps^3 \int_{\Gamma_+} (\beta\cdot\nabla)^2 \ell \beta\cdot n \varphi d\beta d\sigma_x \nonumber\\
    &\quad + \eps^2 \int_\Omega\int_{\mathbb{S}^2} \rho (\beta\cdot\nabla)^2 \ell d\beta dx  + \eps^2 \frac{4\pi}{3} \int_\Omega h \Delta \ell dx  -\eps^2 \int_{\Gamma_+} \beta\cdot n \varphi \beta\cdot \nabla \ell d\beta d\sigma_x  \nonumber \\
    &\quad + \eps \int_\Omega \int_{\mathbb{S}^2} \rho \beta\cdot \nabla \ell d\beta dx+ \int_\Omega\int_{\mathbb{S}^2} \rho \ell d\beta dx \nonumber \\
    &\ge - C \eps^3 \|\varphi-\langle \varphi\rangle/(4\pi)\|_{L^2(\Omega\times\mathbb{S}^2)} \|\ell\|_{H^3(\Omega)} - C\eps^3\|\nabla^2 \ell\|_{L^2(\partial\Omega)}\|\varphi\|_{L^2(\Gamma_+)}\nonumber\\
    &\quad - C\eps^2 \|\rho\|_{L^2(\Omega\times\mathbb{S}^2)}\|\ell\|_{H^2(\Omega)}+ \eps^2 \frac{4\pi}{3} \int_\Omega h \Delta \ell dx - C\eps^2\|\nabla\ell\|_{L^2(\partial\Omega)}\|\varphi\|_{L^2(\Gamma_+)} \nonumber\\
    &\quad - C\eps \|\rho\|_{L^2(\Omega\times\mathbb{S}^2)}\|\ell\|_{H^1(\Omega)} + \int_\Omega\int_{\mathbb{S}^2} \rho \ell d\beta dx,
\end{align*}
where the fact that $\langle (\beta\cdot\nabla)^3 \ell\rangle =0$ is used, which is due to $\ell=\ell(x)$ not depending on $\beta$. By the trace theorem and Sobolev embeddings, the above inequality implies 
\begin{align}\label{eq:est/opA}
    \int_\Omega\mathcal{A}h\cdot \ell dx &\ge \frac{4\pi}{3}\eps^2 \int_\Omega h\Delta \ell dx - C \eps^3 \|\varphi - \langle\varphi\rangle/(4\pi)\|_{L^2(\Omega\times\mathbb{S}^2)} \|\ell\|_{H^3(\Omega)} - C\eps^3 \|\ell\|_{H^3(\Omega)} \|\varphi\|_{L^2(\Gamma_+)} \nonumber\\
    &\quad - C\eps^2\|\ell\|_{H^2(\Omega)}\|\varphi\|_{L^2(\Gamma_+)} - C\eps^2\|\rho\|_{L^2(\Omega\times\mathbb{S}^2)}\|\ell\|_{H^2(\Omega)} - C \eps \|\rho\|_{L^2(\Omega\times\mathbb{S}^2)}\|\ell\|_{H^1(\Omega)} \nonumber\\
    &\quad + \int_\Omega\int_{\mathbb{S}^2}\rho \ell d\beta dx.
\end{align}
Let $\rho=r_2 + r, \varphi=\phi, h =4(T^a)^3 g $ in \eqref{eq:opA},  equation \eqref{eq:l1} can be written as 
\begin{align*}
    \eps^2 \Delta g + \mathcal{A}(4(T^a)^3 g) = r_1  + \langle r\rangle.
\end{align*}
Let $\ell$ be the solution to 
\begin{align*}
    \Delta \ell = g, \text{ in }\Omega,\\
    \ell =0, \text{ on }\partial\Omega.
\end{align*}
We multiply the previous equation by $\ell$ and integration by parts to get 
\begin{align*}
    \varepsilon^2 \int_\Omega g^2 dx &+ \int_\Omega \mathcal{A}(4(T^a)^3 g ) \cdot \ell dx = \int_\Omega  ( r_1  + \langle r \rangle) \ell dx.
\end{align*}
By \eqref{eq:est/opA}, and the estimates for elliptic equations $\|\ell\|_{H^2(\Omega)}\le C \|g\|_{L^{2}(\Omega)}$, $\|\ell\|_{H^3(\Omega)}\le C \|g\|_{H^1(\Omega)}$, we have 
\begin{align*}
    &\int_\Omega \mathcal{A}(4(T^a)^3 g ) \ell dx \nonumber\\
    &\quad \ge \eps^2 \frac{4\pi}{3} \int_\Omega 4(T^a)^3 g^2 dx - C \eps^3 \|\phi-\langle \phi\rangle/(4\pi)\|_{L^2(\Omega\times\mathbb{S}^2)} \|g\|_{H^1(\Omega)} - C\eps^3 \|g\|_{H^1(\Omega)}\|\phi\|_{L^2(\Gamma_+)}\nonumber\\
    &\qquad - C \eps^2 \|g\|_{L^2(\Omega)}\|\phi\|_{L^2(\Gamma_+)} - C\eps (\|r_2\|_{L^2(\Omega\times\mathbb{S}^2)} + \|r\|_{L^2(\Omega\times\mathbb{S}^2)}) \|g\|_{L^2(\Omega)} + \int_\Omega \int_{\mathbb{S}^2} (r_2+r) \ell d\beta dx \nonumber\\
    &\ge \eps^2 \frac{4\pi}{3} \int_\Omega 4(T^a)^3 g^2 dx - C \eps^3 \|\phi - 4(T^a)^3 g\|_{L^2(\Omega\times\mathbb{S}^2)}^2 - C \eps^3 \|g\|_{H^1(\Omega)}^2 - C\eps^3 \|g\|_{H^1(\Omega)}^2 - C\eps^3\|\phi\|_{L^2(\Gamma_+)}^2\nonumber\\
    &\quad -\frac14\eps^2 \|g\|_{L^2(\Omega)}^2- C\eps^2 \|\phi\|_{L^2(\Gamma_+)}^2 -C(\|r\|_{L^2(\Omega\times\mathbb{S}^2)}^2 + \|r_2\|_{L^2(\Omega\times\mathbb{S}^2)}^2)- \frac14 \eps^2 \|g\|_{L^2(\Omega)}^2 \nonumber\\
    &\quad +  \int_\Omega \int_{\mathbb{S}^2} (r_2+r) \ell d\beta dx .
\end{align*}
Taking it into the previous equation gives
\begin{align*}
    &\eps^2 \|g\|_{L^2(\Omega)}^2 + \eps^2 \|2(T^a)^{3/2} g\|_{L^2(\Omega)}^2 \nonumber\\
    &\quad  \le C\eps^3 \|\phi - (4T^a)^3 g\|_{L^2(\Omega\times\mathbb{S}^2)}^2 + C\eps^3 \| g\|_{H^1(\Omega)}^2 + C \eps^2 \|\phi\|_{L^2(\Gamma_+)}^2 + C  (\|r\|_{L^2(\Omega\times\mathbb{S}^2)}^2+ \|r_2\|_{L^2(\Omega\times\mathbb{S}^2)}^2)  \nonumber\\
    &\qquad + \int_\Omega r_1 \ell dx - \int_\Omega \int_{\mathbb{S}^2} r_2 \ell d\beta dx \nonumber\\
    &\quad \le  C\eps^3 \|\phi - (4T^a)^3 g\|_{L^2(\Omega\times\mathbb{S}^2)}^2 + C\eps^3 \| g\|_{H^1(\Omega)}^2 + C \eps^2 \|\phi\|_{L^2(\Gamma_+)}^2 + C  (\|r\|_{L^2(\Omega\times\mathbb{S}^2)}^2+ \|r_2\|_{L^2(\Omega\times\mathbb{S}^2)}^2)  \nonumber\\
    &\qquad + \frac{C}{\eps^2}( \|r_1\|_{L^2(\Omega)}^2 + \|r_2\|_{L^2(\Omega\times\mathbb{S}^2)}^2) + \frac14 \eps^2\|g\|_{L^2(\Omega)}^2.
\end{align*}
Combining this inequality with \eqref{eq:energyest}, we obtain for $\eps$ sufficiently small ($\eps\le \min\{1,\kappa\}/C$),
\begin{align*}
    &\eps^2 \kappa \|\nabla g\|_{L^2(\Omega)}^2 + \eps^2\|2(T^a)^{3/2}g\|_{L^2(\Omega)}^2+\eps \|\phi\|_{L^2(\Gamma_+)}^2 + \|\phi-4(T^a)^3 g\|_{L^2(\Omega\times \mathbb{S}^2)}^2 \nonumber\\
    &\le C \|r\|_{L^2(\Omega\times\mathbb{S}^2)}^2 + \frac{C}{\eps^2}(\|r_1\|_{L^2(\Omega)}^2 + \|r_2\|_{L^2(\Omega\times\mathbb{S}^2)}^2) + \iint_{\Omega\times\mathbb{S}^2} \phi r_2 d\beta dx - \int_\Omega 4(T^a)^3 gr_1 dx + C\|r_1\|_{L^2(\Omega)}^2.
\end{align*}
Using the inequalities 
\begin{align*}
    \int_{\Omega} \int_{\mathbb{S}^2} \phi r_2 d\beta dx &= \int_{\Omega} \int_{\mathbb{S}^2} (\phi - 4(T^a)^3 g) r_2 d\beta dx + \int_\Omega \int_{\mathbb{S}^2} 4(T^a)^3 g r_2 d\beta dx \nonumber\\
    &\le \frac12 \|\phi-4(T^a)^3 g\|_{L^2(\Omega\times \mathbb{S}^2)}^2 + \frac12 \|r_2\|_{L^2(\Omega\times\mathbb{S}^2)}^2 + \frac14 \eps^2 \|2(T^a)^{3/2}g\|_{L^2(\Omega)}^2 + \frac{C}{\eps^2}\|r_2\|_{L^2(\Omega\times\mathbb{S}^2)}^2
\end{align*}
and 
\begin{align*}
    -\int_\Omega 4(T^a)^3 g r_1 dx \le  \frac14 \eps^2 \|2(T^a)^{3/2}g\|_{L^2(\Omega)}^2 + \frac{C}{\eps^2}\|r_1\|_{L^2(\Omega)}^2,
\end{align*}
we get 
\begin{align*}
    & \eps^2 \| g\|_{H^1(\Omega)}^2 + \eps^2 \|2(T^a)^{3/2}g\|_{L^2(\Omega)}^2+\eps \|\phi\|_{L^2(\Gamma_+)}^2 + \|\phi-4(T^a)^3 g\|_{L^2(\Omega\times \mathbb{S}^2)}^2 \nonumber\\
    &\quad \le C\|r\|_{L^2(\Omega\times\mathbb{S}^2)}^2 + \frac{C}{\eps^2}(\|r_1\|_{L^2(\Omega)}^2 + \|r_2\|_{L^2(\Omega\times\mathbb{S}^2)}^2).
\end{align*}
Note that 
\[
\eps^2 \|\phi\|_{L^2(\Omega\times\mathbb{S}^2)}^2 \le \eps^2 \|\phi - 4(T^a)^3 g\|_{L^2(\Omega\times\mathbb{S}^2)}^2 + \eps^2\|4(T^a)^3 g\|_{L^2(\Omega\times\mathbb{S}^2)}^2 \le \eps^2  \|\phi - 4(T^a)^3 g\|_{L^2(\Omega\times\mathbb{S}^2)}^2 + C\eps^2 \|g\|_{L^2(\Omega)}^2,
\]
due to $T^a$ being bounded. Therefore, we arrive at the estimate \eqref{eq:l2result}.

 \bigskip 

 \emph{{Step 3: $L^\infty$ estimate.}} We now derive the $L^\infty$ estimate of $g,\phi$. First, by the maximum principle for linear transport equation, see for example Lemma 3.1 in \cite{wu2015geometric}, the following estimate holds for \eqref{eq:l2}:
 \begin{align*}
     \|\phi\|_{L^\infty(\Omega\times\mathbb{S}^2)} \le \|4(T^a)^3 g\|_{L^\infty(\Omega)} + \|r_2\|_{L^\infty(\Omega\times\mathbb{S}^2)} + \|r\|_{L^\infty(\Omega\times\mathbb{S}^2)}.
 \end{align*}
 Since $\|T^a\|_{L^\infty(\Omega)}$ is bounded,  
 \begin{align}\label{eq:linf1}
    \|\phi\|_{L^\infty(\Omega\times\mathbb{S}^2)} \le C \|g\|_{L^\infty(\Omega)} + \|r_2\|_{L^\infty(\Omega\times\mathbb{S}^2)} + \|r\|_{L^\infty(\Omega\times\mathbb{S}^2)}.
\end{align}
We now give the $L^\infty$ estimate of $g$. Equation \eqref{eq:l1} can be written as 
\begin{align}\label{eq:D2g=f}
    \Delta g = f,
\end{align}
with $f = (-\langle \phi - 4(T^a)^3 g\rangle + r_1+\langle  r \rangle)/\eps^2$.
According to the elliptic regularity, we have 
\begin{align*}
    \|g\|_{L^\infty(\Omega)} \le C(\|g\|_{L^2(\Omega)} + \|f\|_{L^2(\Omega)}).
\end{align*}
combining the above inequality with \eqref{eq:l2result}, we obtain
\begin{align*}
    &\|g\|_{L^\infty(\Omega)} \le C(\|g\|_{L^2(\Omega)} +\|f\|_{L^2(\Omega)}) \\
    &\quad \le \frac{C}{\eps} \|r\|_{L^2(\Omega\times\mathbb{S}^2} + \frac{C}{\eps^2} ( \|r_1\|_{L^2(\Omega)} + \|r_2\|_{L^2(\Omega\times\mathbb{S})})+ \frac{C}{\eps^2} \left(\|\phi-4(T^a)^3 g\|_{L^2(\Omega)} + \|r_1\|_{L^2(\Omega)} + \|r\|_{L^2(\Omega\times\mathbb{S})}\right).
\end{align*}
 Adding  the above inequality with \eqref{eq:linf1} and using \eqref{eq:l2result} on the last term of the above inequality, we arrive to the estimate \eqref{eq:linfresult},
which finishes the proof of Lemma \ref{lm:linear}.
\end{proof}

\begin{rem}
    We follow a similar procedure as \cite{wu2015geometric}. Here the estimate on the elliptic equation \eqref{eq:l1} is needed, whereas in \cite{wu2015geometric} only the estimate on the transport equation is involved. In particular, we derive the $L^2$ and $L^\infty$ energy estimates. Here we only give the estimates on the solutions. However, the existence of solutions can be obtained by using classical existence theory of linear transport equations and elliptic equations.
\end{rem}
\subsection{Nonlinear system}

We now show the existence and uniqueness of solutions to system \eqref{eq:1}-\eqref{eq:2} around the constructed composite approximate solution $(T^a,\psi^a)$ and finish the proof of Theorem \ref{thm.1}. 
\begin{proof}[Proof of Theorem \ref{thm.1}]
    The existence and uniqueness of solutions to system \eqref{eq:1}-\eqref{eq:2} can be obtained by showing the existence for system \eqref{eq:n1}-\eqref{eq:n2}. To achieve this, we construct a sequence of functions $\{g^n,\phi^n\}$ by 
\begin{align*}
    g^0(x) = 0, \quad \phi^0(x,\beta) = 0,
\end{align*}
and for $n\ge 1$,
\begin{align}
    &\eps^2 \Delta g^n + \langle \phi^n - 4(T^a)^3 g^n \rangle = -\mathcal{R}_1 + 4\pi(6(T^a)^2 (g^{n-1})^2 + 4(T^a) (g^{n-1})^3  + (g^{n-1})^4), \label{eq:nn1}\\
    &\eps \beta\cdot \nabla \phi^n +(\phi^n - 4(T^a)^3 g^n) = -\mathcal{R}_2 + 6(T^a)^2 (g^{n-1})^2 + 4(T^a) (g^{n-1})^3  + (g^{n-1})^4,\label{eq:nn2}
\end{align}
with boundary conditions 
\begin{align*}
    &g^n (x)= 0,\text{ for }x\in \partial\Omega,\\
    &\phi^n(x,\beta) = \phi_b(x,\beta),\text{ for } (x,\beta) \in \Gamma_{-}.
\end{align*}
The above relation defines a mapping $\mathcal{T}$ with $(g^n,\phi^n)=\mathcal{T}((g^{n-1},\phi^{n-1}))$. 
Note that the above system is the same with \eqref{eq:tt1}-\eqref{eq:tt2} in the introduction with $T^n = g^n+T^a$, $\psi^n=\phi^n+\psi^a$.

Let $Y=L^2(\Omega)\cap L^\infty(\Omega)$ and $W=L^2(\Omega\times\mathbb{S}^2)\cap L^\infty(\Omega\times\mathbb{S}^2)$. We consider the solution in the function space 
\begin{align*}
    \mathcal{O}_s:=\{(g,\phi)\in Y\times W:\|g\|_{L^2(\Omega)}\le C\eps^s, \|g\|_{L^\infty(\Omega)}\le C\eps^s, \|\phi\|_{L^2(\Omega\times\mathbb{S}^2)}\le C\eps^s, \|\phi\|_{L^\infty(\Omega\times\mathbb{S}^2)} \le C\eps^s\},
\end{align*}
with $s>0$ is a constant to be chosen later. 

First we show $\mathcal{T}$ maps the space $\mathcal{O}_s$ onto itself.
Assume the residuals satisfy $\|\mathcal{R}_1\|_{L^2(\Omega)}$, $\|\mathcal{R}_2\|_{L^2(\Omega\times\mathbb{S}^2)}$, $\|\mathcal{R}_1\|_{L^\infty(\Omega)}$, $\|\mathcal{R}_2\|_{L^\infty(\Omega\times\mathbb{S}^2)} \le  \eps^p$ for some constant $p>0$. Next, we show if $(g^{n-1},\psi^{n-1}) \in \mathcal{O}_{s}$, then $(g^n,\psi^n) \in \mathcal{O}_{s}$. By \eqref{eq:l2result} with $r_1=-\mathcal{R}_1,r_2=-\mathcal{R}_2$ and $r = 6(T^a)^2 (g^{n-1})^2 + 4T^a (g^{n-1})^3 + (g^{n-1})^4$, the following estimate holds:
\begin{align*}
    &\varepsilon \|\phi^n\|_{L^2(\Omega\times\mathbb{S}^2)} + \eps \|g^n\|_{H^1(\Omega)} + \sqrt{\eps}\|\phi^n\|_{L^2(\Gamma_+)} + \|\phi^n-4(T^a)^3 g^n\|_{L^2(\Omega\times\mathbb{S}^2)} \nonumber\\
    &\quad \le \frac{C}{\eps} \left(\|\mathcal{R}_1\|_{L^2(\Omega)} + \|\mathcal{R}_2\|_{L^2(\Omega\times\mathbb{S}^2)}\right)+ C\left(\|6(T^a)^2 (g^{n-1})^2 + 4T^a (g^{n-1})^3 + (g^{n-1})^4\|_{L^2(\Omega)}\right) \nonumber\\
    &\quad \le C \eps^{p-1} + C( \|g^{n-1}\|_{L^4}^2 + \|g^{n-1}\|_{L^8}^4) \nonumber\\
    &\quad \le C \eps^{p-1} + C (\|g^{n-1}\|_{L^\infty(\Omega)} \|g^{n-1}\|_{L^2(\Omega)} + \|g^{n-1}\|_{L^\infty}^3\|g^{n-1}\|_{L^2(\Omega)}) \nonumber\\
    &\quad \le C \eps^{p-1} + C \eps^{2s} .
\end{align*}
Assume $p-2\ge s$ and $2s-1\ge s$, i.e. $p\ge s+2$ and $s\ge 1$, the above inequality implies 
\begin{align*}
    \|g^{n-1}\|_{H^1(\Omega)}, \|\phi^n\|_{L^2(\Omega\times\mathbb{S}^2)}\le C \eps^{p-2} + C \eps^{2s-1} \le C \eps^s.
\end{align*}
Moreover, by \eqref{eq:linfresult}, 
\begin{align*}
    \|\phi\|_{L^\infty(\Omega\times\mathbb{S}^2)} + \|g\|_{L^\infty(\Omega)} &\le \frac{C}{\eps^{2}} \left(\|6(T^a)^2 (g^{n-1})^2 + 4T^a (g^{n-1})^3 + (g^{n-1})^4\|_{L^2(\Omega)}\right) \nonumber\\
    &\quad +  \frac{1}{\eps^3}(\|\mathcal{R}_1\|_{L^2(\Omega)}+ \|\mathcal{R}_2\|_{L^2(\Omega\times\mathbb{S}^2)}  ) + \|\mathcal{R}_2\|_{L^\infty(\Omega\times\mathbb{S}^2)} \nonumber\\
    &\quad +\|6(T^a)^2 (g^{n-1})^2 + 4T^a (g^{n-1})^3 + (g^{n-1})^4\|_{L^\infty(\Omega)} \nonumber\\
    &\le C \eps^{2s-2} + \eps^{p-3} + \eps^p + \eps^{2s} \nonumber\\
    &\le  C \eps^{2s-2} + \eps^{p-3} .
\end{align*}
Assume $p-3\ge s$ and $2s-2\ge s$, i.e. $p\ge s+3$ and $s\ge 2$, the above inequality implies 
\begin{align*}
    \|g^n\|_{L^\infty(\Omega)}, \|\phi^n\|_{L^\infty(\Omega)} \le C \eps^{2s-2} + \eps^{p-3} \le C \eps^s.
\end{align*}
Thus we obtain that $(g^n,\phi^n) \in \mathcal{O}_s$ and therefore $\mathcal{T}$ maps $\mathcal{O}_s$ onto itself.

Next we show the map $\mathcal{T}$
 is a contraction mapping. Let
$h^n = g^{n}-g^{n-1},\varphi^n = \phi^n - \phi^{n-1}$, then they satisfy
\begin{align*}
    &\eps^2 \Delta h^n + \langle \varphi^n - 4(T^a)^3 h^n \rangle = 4\pi f_n,\\
    &\eps \beta\cdot \nabla \varphi^n + \varphi^n - 4(T^a)^3 h^n = f_n,
\end{align*}
where 
\begin{align*}
    f_n &= 6(T^a)^2 (g^{n-1})^2 + 4(T^a) (g^{n-1})^3  + (g^{n-1})^4 - (6(T^a)^2 (g^{n-2})^2 + 4(T^a) (g^{n-2})^3  + (g^{n-2})^4) \nonumber\\
    &= 6(T^a)^2 (g^{n-1}+g^{n-2}) h^{n-1} + 4T^a h^{n-1} ((g^{n-1})^2-g^{n-1}g^{n-2}+(g^{n-2})^2) \nonumber\\
    &\quad + h^{n-1}(g^{n-1}+g^{n-2}) ((g^{n-1})^2+(g^{n-2})^2).
\end{align*}
Using \eqref{eq:l2result} with $r_1=r_2=0$, $r=f_n$, we obtain 
\begin{align*}
    &\eps \|\varphi^n\|_{L^2(\Omega\times\mathbb{S})} + \eps \| h^n\|_{H^1(\Omega)}+\sqrt{\eps} \|\varphi^n\|_{L^2(\Gamma_+)} + \|\varphi^n-4(T^a)^3 h^n\|_{L^2(\Omega\times \mathbb{S}^2)} \nonumber\\
    &\quad \le  \|f_n\|_{L^2(\Omega\times\mathbb{S}^2)} \le C\eps^{s} \|h^{n-1}\|_{L^2(\Omega)},
\end{align*}
hence 
\begin{align*}
    \|h^n\|_{H^1(\Omega)} + \|\varphi^n\|_{L^2(\Omega\times\mathbb{S})}  \le C\eps^{s-1} \|h^{n-1}\|_{L^2(\Omega)}.
\end{align*}
Using \eqref{eq:linfresult} with $r_1=r_2=0$, $r=f_n$, we obtain
\begin{align*}
    &\|\varphi^n\|_{L^\infty(\Omega\times\mathbb{S}^2)}+\|h^n\|_{L^\infty(\Omega)} \nonumber\\
    &\quad \le \frac{C}{\eps^{2}} \|f_n\|_{L^2(\Omega\times\mathbb{S}^2)}+C\|f_n\|_{L^\infty(\Omega)} \le C\eps^{s-2}\|h^{n-1}\|_{L^2(\Omega)} + C \eps^s \|h^{n-1}\|_{L^\infty(\Omega)}.
\end{align*}
Assume $s\ge 3$, then $C\eps^{s-2}, C \eps^s < 1$ for $\eps$ sufficiently small, the above two inequalities imply
\begin{align*}
    \|h^n\|_Y + \|\varphi^n\|_W \le C_1 (\|h^{n-1}\|_Y + \|\varphi^{n-1}\|_W )
\end{align*}
for some constant $0<C_1<1$. Therefore, for $p\ge s+3$ and $s\ge 3$, 
$\mathcal{T}$
 is a contraction mapping. By the Banach fixed point theorem, there exists a unique fixed point $(g,\phi)$ such that $(g,\phi)=\mathcal{T}((g,\phi))$. Therefore there exists a unique solution to \eqref{eq:n1}-\eqref{eq:n2} in $\mathcal{O}_s$.

Taking $s=3$ and $p=6$, we can conclude that 
\begin{align*}
    \|g\|_{Y} + \|\varphi\|_W \le C \eps^3.
\end{align*}
Note that in order to obain $\mathcal{R}_1,\mathcal{R}_2=O(\eps^6)$, we need to take $n=5$ in the expansion and by Lemma \ref{thm-ap}, 
$    \|\mathcal{R}_1\|_{L^2\cap L^\infty(\Omega)}, ~
\|\mathcal{R}_2\|_{L^2\cap L^\infty(\Omega\times\mathbb{S}^2)} \le C\eps^6$. We have 
\begin{align*}
    &\left\|T^\eps - \sum_{k=0}^5 \eps^k T_k - \sum_{k=0}^{5}\eps^k \bar{T}_k \right\|_{L^2\cap L^\infty(\Omega)} \le C \eps^3, \\
    &\left\|\psi^\eps - \sum_{k=0}^5 \eps^k \psi_k - \sum_{k=0}^{5} \eps^k\bar{\psi}_k \right\|_{L^2\cap L^\infty(\Omega\times\mathbb{S}^2)} \le C \eps^3.
\end{align*}
Therefore, we get 
\begin{align*}
    &\|T^\eps - T_0 - \bar{T}_0\|_{L^\infty(\Omega)} \le C\eps,\\
    &\|\psi^\eps - T_0^4 - \bar{\psi}_0 \|_{L^\infty(\Omega\times\mathbb{S}^2)} \le C\eps,
\end{align*}
which are \eqref{eq:thm} and finish the proof.
\end{proof}

\appendix 
\section{Existence of the steady state radiative transfer system}
Next we prove the existence for the steady state radiative transfer system \eqref{eq:1}-\eqref{eq:2} with boundary conditions \eqref{eq:1b}-\eqref{eq:2b}.
\begin{theorem}
    Assume $\gamma_1\le T_b \le \gamma_2$ and $\gamma_1^4 \le \psi_b \le \gamma_2^4$ for some constants $0\le \gamma_1\le \gamma_2$. There there exists a weak solution $(T^\eps,\psi^\eps)\in L^\infty(\Omega)\times L^\infty(\Omega\times\mathbb{S}^2)$ to \eqref{eq:1}-\eqref{eq:2}. 
\end{theorem}
\begin{proof}
    We show the existence by using the fixed-point theorem. Let $\mathcal{A}:=\{T^\eps \in L^\infty(\Omega):\gamma_1\le T^\eps \le \gamma_2\}$. We define the operator $\mathcal{F}: \mathcal{A}\to \mathcal{A}$ with $\theta=\mathcal{F}T$ by solving 
    \begin{equation}\label{eq:a//1}
    \begin{aligned}
        &\eps\beta\cdot \nabla \psi + \psi = T^4, \\
        &\psi^\eps(x,\beta) = \psi_b(x,\beta),\quad \text{for }(x,\beta)\in\Gamma_-,
    \end{aligned}
\end{equation}
    and 
    \begin{equation}\label{eq:a//2}
    \begin{aligned}
        &\eps^2\Delta \theta - 4\pi^2 \theta^4 = -\langle \psi\rangle,\\
        &\theta(x)=T_b(x),\quad \text{for }x\in \partial\Omega.
    \end{aligned}
\end{equation}
    Next we show if $\gamma_1\le T\le \gamma_2$, then $\gamma_1\le \theta\le \gamma_2$. First the maximum principle for the transport equation implies 
    \begin{align}
        \gamma_1^4 \le \psi \le \gamma_2^4.
    \end{align}
    The maximum principle for equation \eqref{eq:a//2} also implies $\gamma_1\le \theta\le \gamma_2$. Suppose $\theta$ reaches its maximum at $x_M\in\Omega$, then if $x_M\in \partial\Omega$, $\theta(X_M)=T_b \le \gamma_2$ and thus $\theta(x)\le \theta(x_M)\le \gamma_2$ for any $x\in\Omega$. Otherwise if $x_M$ is an interior point, then $\Delta\theta(x_M) \le 0$, and so 
    \begin{align}
        4\pi^2\theta^4(x_M) \le \langle\psi(x_M,\cdot)\rangle \le 4\pi^2\gamma_2^4,
    \end{align}
    hence $\theta(x)\le \theta(x_M) \le \gamma_2$. Using a similar contradiction argument, $\theta(x)\ge \gamma_1$ can be shown.

    Since $\mathcal{F}$ maps $\mathcal{A}$ to itself and $\mathcal{A}$ is a convex compact subset of the Banach space $L^\infty$. Hence by Schauder's fixed point theorem, there exists a fixed point $(T^\eps,\psi^\eps)$ of $\mathcal{F}$.  Since this fixed point satisfies \eqref{eq:a//1} and \eqref{eq:a//2}, hence $(T^\eps,\psi^\eps)$ is a solution to \eqref{eq:1}-\eqref{eq:2} with boundary conditions \eqref{eq:1b}-\eqref{eq:2b}.

\end{proof}

\begin{rem}
    Unlike the time dependent case, where the uniqueness can be shown by showing $\mathcal{F}$ is a contraction mapping (with time step small), we are not able to show $\mathcal{F}$ in the above proof is a contraction mapping and thus uniqueness is not obtained here.
\end{rem}

\section*{Acknowledgement}
The work of N.M. is supported by NSF grant DMS-1716466 and by Tamkeen under the NYU Abu Dhabi Research Institute grant of the center SITE. The work of M.G is supported by Tamkeen under the NYU Abu Dhabi Research Institute grant of the center SITE.

\bibliographystyle{siam}
\bibliography{Referencepaper2}

\end{document}